\documentclass[12pt]{article}%
\usepackage{amsmath}
\usepackage{amsfonts}
\usepackage{amssymb}
\usepackage{graphicx}%
\providecommand{\U}[1]{\protect\rule{.1in}{.1in}}
\newtheorem{theorem}{Theorem}

\newtheorem{condition}[theorem]{Condition}

\newtheorem{corollary}[theorem]{Corollary}

\newtheorem{definition}[theorem]{Definition}
\newtheorem{example}[theorem]{Example}

\newtheorem{lemma}[theorem]{Lemma}

\newtheorem{proposition}[theorem]{Proposition}
\newtheorem{remark}[theorem]{Remark}

\newenvironment{proof}[1][Proof]{\noindent\textbf{#1.} }{\ \rule{0.5em}{0.5em}}
\ifx\pdfoutput\relax\let\pdfoutput=\undefined\fi
\newcount\msipdfoutput
\ifx\pdfoutput\undefined\else
\ifcase\pdfoutput\else
\ifx\paperwidth\undefined\else
\ifdim\paperheight=0pt\relax\else\pdfpageheight\paperheight\fi
\ifdim\paperwidth=0pt\relax\else\pdfpagewidth\paperwidth\fi
\fi\fi\fi
\begin{document}

\title{Notes on the Dirichlet problem of a class of second order elliptic partial
differential equations on a Riemannian manifold}
\date{}
\maketitle

$\bigskip%
\begin{array}
[c]{ccc}%
\text{Jaime Ripoll} & \text{ } & \text{Friedrich Tomi}\\%
\begin{array}
[c]{c}%
\text{Federal Universities of R. G. do Sul}\\
\text{and Santa Maria}%
\end{array}
&  & \text{Heidelberg University}\\
\text{Brazil} &  & \text{Germany}\\
&  &
\end{array}
$%

\[
\text{\textbf{Abstract}}%
\]

In these notes we study the Dirichlet problem for critical points of a convex
functional of the form%
\[
F(u)=\int_{\Omega}\phi\left(  \left\vert \nabla u\right\vert \right)  ,
\]
where $\Omega$ is a bounded domain of a complete Riemannian manifold
$\mathcal{M}.$ We also study the asymptotic Dirichlet problem when
$\Omega=\mathcal{M}$ is a Cartan-Hadamard manifold. Our aim is to present a
unified approach to this problem which comprises the classical examples of the
$p-$Laplacian ($\phi(s)=s^{p}$, $p>1)$ and the minimal surface equation
($\phi(s)=\sqrt{1+s^{2}}$). Our approach does not use the direct method of the
Calculus of Variations which seems to be common in the case of the
$p-$Laplacian. Instead, we use the classical method of a-priori $C^{1}$
estimates of smooth solutions of the Euler-Lagrange equation. These estimates
are obtained by a coordinate free calculus. Degenerate elliptic equations like
the $p-$Laplacian are dealt with by an approximation argument.

These notes address mainly researchers and graduate students interested in
elliptic partial differential equations on Riemannian manifolds and may serve
as a material for corresponding courses and seminars.

\newpage%

\tableofcontents
%

\renewcommand{\contentsname}{SUM\'ARIO}%

\newpage

\section{\label{intr}Introduction}

\qquad In these notes we study the Dirichlet problem for critical points of a
convex functional of the form%
\begin{equation}
F(u)=\int_{\Omega}\phi\left(  \left\vert \nabla u\right\vert \right)  ,
\label{func}%
\end{equation}
where $\Omega$ is a bounded domain of a complete Riemannian manifold
$\mathcal{M}.$ We also study the asymptotic Dirichlet problem when
$\Omega=\mathcal{M}$ is a Cartan-Hadamard manifold.

As minimal conditions on $\phi$ we require that%
\begin{equation}
\left\{
\begin{array}
[c]{l}%
\phi\in C^{1}\left(  \left[  0,\infty\right)  \right)  \cap C^{2}\left(
\left(  0,\infty\right)  \right) \\
\phi^{\prime}(s)>0\text{ and }\phi^{\prime\prime}(s)>0\ \text{for }s>0.
\end{array}
\right.  \label{mc}%
\end{equation}
These conditions imply the strict convexity of $F$ and ensure the ellipticity
of the associated Euler-Lagrange equation.

There is a vast literature on this class of problems, mainly on the Euclidean
space, which we do not discuss here. Our aim is to present a unified approach
to this problem, in the Riemannian setting, which comprises the classical
examples of the $p-$Laplacian ($\phi(s)=s^{p}$, $p>1)$ and the minimal surface
equation ($\phi(s)=\sqrt{1+s^{2}}$). Our approach does not use the
minimization technique of the Calculus of Variations which seems to be common
in the case of the $p-$Laplacian. Instead we use the classical method of
a-priori $C^{1}$ estimates which are obtained from the Euler-Lagrange equation
using a coordinate free calculus. Degenerate elliptic equations like the
$p-$Laplacian are dealt with by an approximation argument.

The $p-$energy and the area are typical representatives for two classes of
functionals which we shall distinguish in what follows. With the abbreviation
$a=\phi^{\prime}$ the Euler-Lagrange equation of $F$ is%
\begin{equation}
Q\left[  u\right]  :=\operatorname{div}\left(  \frac{a(|\nabla u|)}{|\nabla
u|}\nabla u\right)  =0, \label{PDE1}%
\end{equation}
which may be written in the equivalent form%
\begin{equation}
\left\vert \nabla u\right\vert ^{2}\Delta u+b\left(  \left\vert \nabla
u\right\vert \right)  \nabla^{2}u\left(  \nabla u,\nabla u\right)  =0
\label{PDE2}%
\end{equation}
where%
\begin{equation}
b(s)=\frac{sa^{\prime}(s)}{a(s)}-1 \label{be}%
\end{equation}
and $\nabla^{2}$ denotes the Hessian. It follows from (\ref{mc}) that
$1+b(s)>0$ for $s>0.$

As it is well known from the theory of elliptic equations, the behavior of the
eigenvalues of the quadratic form associated with (\ref{PDE2})%
\begin{equation}
q\left(  \xi,\xi\right)  =\left\vert \nabla u\right\vert ^{2}\left\vert
\xi\right\vert ^{2}+b\left(  \left\vert \nabla u\right\vert \right)
\left\langle \xi,\nabla u\right\rangle ^{2} \label{qf}%
\end{equation}
is crucial. Precisely, it is the quotient of the eigenvalue $\lambda$ in
direction $\nabla u$ given by
\[
\lambda=\left\vert \nabla u\right\vert ^{2}\left(  1+b\left(  \left\vert
\nabla u\right\vert \right)  \right)
\]
and the maximal eigenvalue given by%
\[
\Lambda=\left\vert \nabla u\right\vert ^{2}\max\left\{  1,1+b\left(
\left\vert \nabla u\right\vert \right)  \right\}
\]
which is decisive. We may easily see that
\[
\frac{\lambda}{\Lambda}=1+b^{-}%
\]
where $b^{-}=\min\left\{  b,0\right\}  .$ The construction of barriers at the
boundary depends on the behavior of the function $1+b^{-}.$ We consider the
two following possibilities:

\begin{itemize}
\item[Condition I] \emph{Mild decay of the eigenvalue ratio}:%
\begin{equation}
\left(  1+b^{-}(s)\right)  s^{2}\geq\varphi(s),\text{ }s\geq s_{0}>0
\label{fi}%
\end{equation}
where $\varphi$ is non decreasing and%
\begin{equation}
\int_{s_{0}}^{\infty}\frac{\varphi(s)}{s^{2}}ds=+\infty\label{condp}%
\end{equation}

\item[Condition II] \emph{Strong decay of the eingenvalue ratio:}%
\[
\left(  1+b^{-}(s)\right)  s^{2}\geq\varphi(s),\text{ }s\geq s_{0}>0
\]
where $\varphi$ is non increasing and%
\begin{equation}
\int_{s_{0}}^{\infty}\frac{\varphi(s)}{s}ds=+\infty. \label{condm}%
\end{equation}

\end{itemize}

As we will see, with mild decay of the eingenvalue ratio one is able to
construct barriers on arbitrary bounded smooth domains. However, for partial
differential equations with strong decay of the eigenvalue ratio, it is
necessary to require the mean convexity of the domain (see Section \ref{bge}).

Let us mention that the $p-$Laplace equation falls into class I and the
minimal surface equation into class II. Further conditions besides I and II
will have to be imposed for global and local gradient estimates (see Section
\ref{glge}). Let us also mention that the behavior of type I was introduced by
Serrin as \emph{regularly elliptic }(see \cite{S}).

These notes address mainly researchers and graduate students interested in
elliptic partial differential equations on Riemannian manifolds and may serve
as a material for corresponding courses and seminars. Indeed, the first author
gave a course at Federal University of Santa Maria, Rio Grande do Sul, Brazil,
on the second semester of 2017, based on these notes.

Our goal in this text was to carve out structural conditions on the integrand
$\phi$ which lead to global and local $C^{1}-$estimates for solutions of the
corresponding Euler-Lagrange equations. The text gives a complete,
self-contained presentation of this part of the theory; no prerequisites
besides elementary Riemannian geometry are required. Once the crucial $C^{1}%
-$estimates are established a general machinery may be applied to obtain
higher order estimates. For this machinery we refer to the literature
\cite{GT}, \cite{L}, it is not subject of this text.

We believe that the techniques of these notes can be extended to more general
partial differential equations, such as equations with a nonzero right hand
side $Q=f$ with $f$ depending on the point of the manifold, the function and
its first derivatives.

We would like to express our thanks to Roberto Nu\~{n}es for checking part of
this manuscript and for contributing the useful estimate in Remark \ref{rob}.

\newpage

\section{Overview of the technique}

\qquad We resume in this section the main ideas used in these notes to
investigate the Dirichlet problem in bounded smooth domains of a Riemannian manifold.

\subsection{\label{MAPB}The Method of a priori Bounds}

\qquad The case that the partial differential equation (\ref{PDE1}) is
singular or degenerate, as the $p-$Laplacian and a similar family of partial
differential equations, is reduced to the regular case by a perturbation
technique (see Section \ref{et}). The main and largest part of our notes
concerns the existence of solutions of regular partial differential equations.
Regular means that (\ref{PDE1}) is elliptic and has at least
H\"{o}lder-continuous coefficients. To be more precise we write (\ref{PDE1})
in the equivalent form
\[
A\left(  \left\vert \nabla u\right\vert \right)  \Delta u+\frac{A^{\prime
}\left(  \left\vert \nabla u\right\vert \right)  }{\left\vert \nabla
u\right\vert }\nabla^{2}u\left(  \nabla u,\nabla u\right)  =0,\text{
}a(s)=sA(s).
\]
In terms of a local orthonormal frame $E_{1},...,E_{n}$ this equation may be
written as
\[
\sum_{i,j}a_{ij}\left(  \left\vert \nabla u\right\vert \right)  \left(
\nabla^{2}u\right)  _{ij}=0
\]
with%
\[
a_{ij}=A\left(  \left\vert \nabla u\right\vert \right)  \delta_{ij}%
+\frac{A^{\prime}\left(  \left\vert \nabla u\right\vert \right)  }{\left\vert
\nabla u\right\vert }u_{i}u_{j},\text{ }\nabla u=u_{i}E_{i}.
\]
By an elementary but careful computation one sees that for $u\in C^{2}\left(
\overline{\Omega}\right)  $ the coefficients $a_{ij}$ are $\alpha-$H\"{o}lder
continuous provided that $A\in C^{1,\alpha}\left(  \left[  0,\infty\right)
\right)  .$ Moreover, the eigenvalues of the matrix $\left(  a_{ij}\right)  $
are $A(s)$ and $A(s)+sA^{\prime}(s)$ so that ellipticity amounts to the
inequalities%
\[
A(s)>0,\text{ }A(s)+sA^{\prime}(s)>0
\]
for all $s\geq0.$

To investigate the Dirichlet problem for our partial differential equations
which are regular, for \emph{smooth boundary data}, we use the classical
method of a priori bounds. In the abstract setting of Functional Analysis,
this method is conveniently exposed in form of the following fixed point
theorem of Leray-Schauder.

Let $\mathcal{B}$ be a Banach space and denote by $\left\Vert \text{
}\right\Vert $ the norm of $\mathcal{B}$. A continuous mapping
$T:\mathcal{B\rightarrow B}$ is called compact if the image by $T$ of bounded
subsets of $\mathcal{B}$ are precompact that is, their closures are compact
subsets of $\mathcal{B}.$

\begin{theorem}
\label{ls}Let $T:\mathcal{B\rightarrow B}$ be a compact mapping and set
\begin{equation}
V:=\left\{  v\in\mathcal{B}\text{
$\vert$
}v=\sigma T\left(  v\right)  \text{ for some }\sigma\in\left[  0,1\right]
\right\}  . \label{V}%
\end{equation}
Assume that there is a constant $C$ such that
\[
\left\Vert v\right\Vert _{\mathcal{B}}\leq C
\]
for all $v\in V.$ Then $T$ has a fixed point that is, there is $u\in
\mathcal{B}$ such that $T\left(  u\right)  =u.$
\end{theorem}

This theorem may lastly be derived from Brouwer fixed point theorem (which
asserts that a continuous mapping from a ball of $\mathbb{R}^{n}$ into itself
has a fixed point, see Chapter 11 of \cite{GT}). We now show how one can use
Theorem \ref{ls} for investigating the existence of solutions to the Dirichlet
problem%
\begin{equation}
\left\{
\begin{array}
[c]{l}%
Q\left[  u\right]  =A\left(  \left\vert \nabla u\right\vert \right)  \Delta
u+\frac{A^{\prime}\left(  \left\vert \nabla u\right\vert \right)  }{\left\vert
\nabla u\right\vert }\nabla^{2}u\left(  \nabla u,\nabla u\right)  =0\text{ in
}\Omega\\
u|\partial\Omega=g
\end{array}
\right.  \label{dirr}%
\end{equation}
for a given $g\in C^{2,\alpha}\left(  \overline{\Omega}\right)  ,$ where
$\Omega$ is a $C^{2,\alpha}$ fixed bounded domain of $\mathcal{M}$, $A\left(
\left\vert \nabla u\right\vert \right)  =a\left(  \left\vert \nabla
u\right\vert \right)  /\left\vert \nabla u\right\vert .$ We assume that $Q$ is regular.

In order to apply Theorem \ref{ls} we take $\mathcal{B}=C^{2}\left(
\overline{\Omega}\right)  $ and define the operator $T:C^{2}\left(
\overline{\Omega}\right)  \rightarrow C^{2}\left(  \overline{\Omega}\right)  $
as follows. Any $u\in C^{2}\left(  \overline{\Omega}\right)  $ determines a
linear operator%
\[
\mathcal{L}_{u}\left[  v\right]  =A\left(  \left\vert \nabla u\right\vert
\right)  \Delta v+\frac{A^{\prime}\left(  \left\vert \nabla u\right\vert
\right)  }{\left\vert \nabla u\right\vert }\nabla^{2}v\left(  \nabla u,\nabla
u\right)
\]
which satisfies $\mathcal{L}_{u}\left[  u\right]  =0$ if and only if $Q\left[
u\right]  =0.$ From the regularity assumption of $Q$, $\mathcal{L}_{u}$ is
elliptic and has H\"{o}lder coefficients. It then follows from Theorem 6.14 of
\cite{GT} that the Dirichlet problem%
\[
\left\{
\begin{array}
[c]{l}%
\mathcal{L}_{u}\left[  w\right]  =0\text{ in }\Omega\\
w|\partial\Omega=g
\end{array}
\right.
\]
has a unique solution $w\in C^{2,\alpha}\left(  \overline{\Omega}\right)  .$
We may then define $T\left[  u\right]  =w\Leftrightarrow\mathcal{L}_{u}\left[
w\right]  =0$ and $w|\partial\Omega=g.$ It is clear that $u\in C^{2}\left(
\overline{\Omega}\right)  $ is a solution of (\ref{dirr}) if and only if $u$
is a fixed point of $T.$

To prove that $T$ is a compact operator we first note that $\mathcal{L}_{u}$
satisfies the maximum principle (Theorem 3.5 of \cite{GT}) which implies, in
combination with Theorem 6.6 of \cite{GT}, that if $w\in C^{2.\alpha}\left(
\overline{\Omega}\right)  $ satisfies $\mathcal{L}_{u}\left[  w\right]  =0$
then there is a constant $C$ depending only on a $C^{2}\left(  \overline
{\Omega}\right)  $ bound for $u$ such that%
\[
\left\vert w\right\vert _{C^{2,\alpha}}\leq C\left\vert g\right\vert
_{C^{2,\alpha}}.
\]
Since the embedding of $C^{2,\alpha}\left(  \overline{\Omega}\right)  $ into
$C^{2}\left(  \overline{\Omega}\right)  $ is compact it follows that $T$ maps
bounded subsets into precompact subsets. The continuity of $T$ follows by a
similar argument.

Now, let $v\in V$ where $V$ is given by (\ref{V}). We may assume that $v$ is
nonzero. Then there is $\sigma\in\left(  0,1\right]  $ such that $v=\sigma
T\left[  v\right]  .$ Obviously $v/\sigma\in C^{2,\alpha}\left(
\overline{\Omega}\right)  $ and $\mathcal{L}_{v}\left[  v/\sigma\right]
=\mathcal{L}_{v}\left[  T\left[  v\right]  \right]  =0.$ Hence, by the
linearity of $\mathcal{L}_{v},$ $\mathcal{L}_{v}\left[  v\right]
=\sigma\mathcal{L}_{v}\left[  v/\sigma\right]  =0$ that is, $v$ satisfies
$Q\left[  v\right]  =0$ with $v|_{\partial\Omega}=\sigma g.$ Hence, the
applicability of Theorem \ref{ls} depends on obtaining uniform estimate of the
$C^{2}$ norm of $v$ that is, an upper bound of $\left\vert v\right\vert
_{C^{2}}$ depending only on $\left\vert g\right\vert _{C^{2,\alpha}\left(
\overline{\Omega}\right)  }$ (besides $\Omega$ and $a,$ but not on $\sigma)$.
This is what one calls in partial differential equations theory as
\textquotedblleft a priori\textquotedblright\ estimates since they can be
obtained independently of the existence of a solution of (\ref{dirr}). How one
obtains these estimates in our case is a matter of a preliminary discussion in
the next section.

\subsection{A priori estimates}

\qquad We assume here that $Q$ is regular and that $\Omega$ is a $C^{2,\alpha
}$ bounded domain. A fundamental tool to obtain a priori estimates is the
\emph{comparison principle. }It says that if $v,w\in C^{2}\left(
\Omega\right)  \cap C^{0}\left(  \overline{\Omega}\right)  $ are sub and
supersolutions respectively of $Q$ (that is $Q(u)\geq0$ and $Q(w)\leq0,$ in
the classical or weak sense, see Section \ref{cp} for details), and if
$v|_{\partial\Omega}\leq w|_{\partial\Omega}$ then $v\leq w$ in $\Omega$. In
our case, once we prove a comparison principle for our partial differential
equations, since the constant functions are solutions of $Q\left[  u\right]
=0$ we immediately have%
\[
\inf_{\Omega}g\leq u\leq\sup_{\Omega}g
\]
if $u\in C^{2}\left(  \Omega\right)  \cap C^{0}\left(  \overline{\Omega
}\right)  $ is a solution of (\ref{dirr}). We then have an a priori estimate
for the $C^{0}$ norm%
\begin{equation}
\left\vert u\right\vert _{C^{0}}\leq\sup_{\Omega}\left\vert g\right\vert
\label{c0apri}%
\end{equation}
for any possible solution $u\in C^{2}\left(  \Omega\right)  \cap C^{0}\left(
\overline{\Omega}\right)  $ of (\ref{dirr}).

\subsubsection{\label{mapbs}The case of smooth boundary data}

\qquad We begin by obtaining a priori gradient estimates, first by using
barriers to control the gradient at the boundary of the domain. Barriers at a
given point $p\in\partial\Omega$ are sub and supersolutions $v,w\in
C^{1}\left(  \overline{\Omega}\right)  $ of $Q,$ respectively, such that
\[
v(p)=w(p)=g(p)
\]
and
\[
v(x)\leq g(x)\leq w(x)
\]
for all $x\in\partial\Omega.$ From the comparison principle it then follows
that if $u\in C^{1}\left(  \overline{\Omega}\right)  $ is a solution of
(\ref{dirr}) then
\[
v\leq u\leq w
\]
and, from elementary analysis, it follows that
\[
\left\vert \nabla u(p)\right\vert \leq\max\{|\nabla w(p)|,|\nabla v(p)|\}.
\]

In many cases, when $\mathcal{M}$ has nonnegative Ricci curvature a maximum
principle holds for the norm of the gradient, that is%
\[
\sup_{\Omega}\left\vert \nabla u\right\vert =\sup_{\partial\Omega}\left\vert
\nabla u\right\vert
\]
from which, with the help of barriers, one obtains an a priori $C^{1}$
estimates of a solution of (\ref{dirr}). In a general manifold, where the
Ricci curvature can be negative or change sign, and for the general class of
partial differential equations considered here, the maximum of the gradient
can, in principle, occur at an interior point of the domain. If this happens
we prove that the gradient at such a point is controlled by the $C^{0}$ norm
of the solution. Together with the boundary gradient estimate we then obtain
an a priori $C^{1}$ estimate of a solution of (\ref{dirr}).

Next we need H\"{o}lder estimates for the gradient. For this step we refer to
well established theories in the literature (see \cite{GT}, Theorem 13.2). We
may then assert that there is $\gamma>0$ such that $u$ has an a priori
$C^{1,\gamma}$ norm bound in $\overline{\Omega}$ $($with $\gamma$ depending
only on $\left\vert u\right\vert _{C^{1}}$ and hence only on $\left\vert
g\right\vert _{C^{2,\alpha}})$ and therefore the coefficients of the operator
$\mathcal{L}_{u}$ have uniformly bounded norm in $C^{\alpha\gamma}\left(
\overline{\Omega}\right)  .$ We may then again apply the linear theory (see
Theorem 6.6 of \cite{GT}) to the equation%
\[
\mathcal{L}_{u}\left[  u\right]  =0,\text{ }u|_{\partial\Omega}=\sigma g
\]
to obtain an a priori $C^{2}\left(  \overline{\Omega}\right)  $ bound for $u.$

We have seen that the solvability of the Dirichlet problem (\ref{dirr})
reduces to proving a comparison principle for $Q$ and obtaining global a
priori estimates for the gradient. This will be done in the next sections. It
remains to consider the Dirichlet problem (\ref{dirr}) for continuous boundary
data. The main points of this are discussed in the next section.

\subsubsection{\label{mapbc}The case of continuous boundary data}

\qquad We assume that $Q$ is regular and that (\ref{dirr}) is solvable for
$C^{2,\alpha}$ boundary data. To study the problem (\ref{dirr}) in the case
that $\Omega$ is bounded and of $C^{2,\alpha}$ class but $g$ is only
continuous we need to obtain a priori \emph{local }gradient estimates. We
begin by obtaining a maximum principle for the difference of two solutions
(Proposition \ref{maxp}) that is, if $u,v\in C^{2}\left(  \Omega\right)  \cap
C^{0}\left(  \overline{\Omega}\right)  $ satisfy $Q\left[  u\right]  =Q\left[
v\right]  =0$ in $\Omega$ then%
\begin{equation}
\sup_{\Omega}\left\vert u-v\right\vert =\sup_{\partial\Omega}\left\vert
u-v\right\vert . \label{max}%
\end{equation}

We then consider a sequence $g_{k}\in C^{2,\alpha}\left(  \overline{\Omega
}\right)  $ converging in the $C^{0}$ norm to $g$. From the maximum principle
the corresponding solutions $u_{k}\in C^{2}\left(  \overline{\Omega}\right)  $
to the Dirichlet problem with boundary data $g_{k}$ satisfy%
\[
\sup_{\Omega}\left\vert u_{k}-u_{j}\right\vert =\sup_{\Omega}\left\vert
g_{k}-g_{j}\right\vert ,\text{ }j,k\in\mathbb{N},
\]
being hence a Cauchy sequence on $C^{0}\left(  \overline{\Omega}\right)  .$
The sequence $\left(  u_{k}\right)  $ then converges in the $C^{0}$ norm to
some $u\in C^{0}\left(  \overline{\Omega}\right)  .$ To prove that in fact
$u\in C^{2}\left(  \Omega\right)  $ we obtain local gradient estimates\emph{.
}This is done by fixing an arbitrary $x\in\Omega$ and choosing $r>0$ such that
the closed geodesic ball $B_{r}(x)$ is contained in $\Omega$ and is normal.
Then we prove that there is a constant $C$ depending only on $a,$ $r,g$ such
that%
\[
\left\vert \nabla u_{k}(x)\right\vert \leq C
\]
for all $k\in\mathbb{N}.$ Then, as before, we can make use of Theorem 13.2 of
\cite{GT} to get an uniform $C^{1,\gamma}-$norm bound of $\left(
u_{k}\right)  $ for some $\gamma>0\ $and the linear theory (Theorem 6.6 of
\cite{GT}) to prove that the sequence $\left(  u_{k}\right)  $ has equibounded
$C^{2,\gamma\alpha}$ norm on any relatively compact subdomain of $\Omega.$
Therefore, it contains a subsequence converging in the $C^{2}$ norm to $u$ on
any such subdomain and thus $u\in C^{2}\left(  \Omega\right)  .$ By the
continuity of $Q:C^{2}\left(  \Omega\right)  \rightarrow C^{0}\left(
\Omega\right)  $ it follows that $Q\left[  u\right]  =0.$

The boundary gradient estimates, the gradient estimates at interior points for
smooth boundary data, and the local gradient estimates are obtained in the
next sections.

\newpage

\section{Gradient estimates}

\qquad In this section we derive global and local estimates for solutions for
regular partial differential equations under the assumption that the solutions
are of $C^{3}$ class.

\subsection{\label{cp}Comparison and maximum principles}

Let $\Omega$ be a domain in $\mathcal{M}$, $\overline{\Omega}$ compact. We
assume that
\[
Q\left[  u\right]  =\operatorname{div}\left(  \frac{a\left(  \left\vert \nabla
u\right\vert \right)  }{\left\vert \nabla u\right\vert }\nabla u\right)
\]
is such that $a:\left[  0,\infty\right)  \rightarrow\mathbb{R}$ is strictly
increasing and $a(0)=0.$

Denote by $C_{0}^{0,1}(\Omega)$ the space of Lipschitz functions which compact
support on $\Omega.$ We say that $u\in C^{0,1}(\Omega)\cap C^{0}\left(
\overline{\Omega}\right)  $ is a weak solution of $Q$ if%
\begin{equation}
\int_{\Omega}\left\langle \frac{a\left(  |\nabla u|\right)  }{|\nabla
u|}\nabla u,\nabla\xi\right\rangle dx=0 \label{ws}%
\end{equation}
for all $\xi\in C_{0}^{0,1}(\Omega).$ We say that $v\in C^{0,1}(\Omega)\cap
C^{0}\left(  \overline{\Omega}\right)  $ is a weak supersolution (subsolution)
of $Q$ if (\ref{ws}) holds with \textquotedblleft$\geq$\textquotedblright%
\ (\textquotedblleft$\leq$\textquotedblright) instead of \textquotedblleft%
$=$\textquotedblright\ for all $\xi\in C_{0}^{0,1}(\Omega)$ with $\xi\geq0$
a.e$.$ on $\Omega.$

\begin{proposition}
[Comparison Principle]\label{QvleQu} Let $\Omega\subset\mathcal{M}$ be an open
set, $u\in C^{0,1}(\Omega)\cap C^{0}\left(  \overline{\Omega}\right)  $ a weak
subsolution of $Q$ and $v\in C^{0,1}(\Omega)\cap C^{0}\left(  \overline
{\Omega}\right)  $ a weak supersolution of $Q\ $such that
\begin{equation}
\lim\sup_{k}\left(  u(x_{k})-v(x_{k})\right)  \leq0 \label{bcp}%
\end{equation}
for any sequence $x_{k}$ in $\Omega$ which leaves any compact subset of
$\Omega.$ Then it follows that $u\leq v$ in $\Omega.$
\end{proposition}

\begin{proof}
Let $\varepsilon>0$ and let us choose
\[
\zeta:=\left(  u-v-\varepsilon\right)  ^{+}=\max\left\{  u-v-\varepsilon
,0\right\}  .
\]
From (\ref{bcp}) it follows that $u$ and $v$ have bounded first derivatives on
the support of $\zeta,$ a compact subset of $\Omega,$ and
\[
\nabla\zeta=\left\{
\begin{array}
[c]{l}%
\nabla\left(  u-v\right)  \text{ if }u-v>\varepsilon\\
0\text{ elsewhere.}%
\end{array}
\right.
\]
Therefore $\left\vert \nabla u\right\vert $ and $\left\vert \nabla
v\right\vert $ are integrable on the set
\[
\Lambda_{\varepsilon}:=\left\{  x\in\Omega\text{
$\vert$
}u(x)-v(x)>\varepsilon\right\}
\]
and we have%

\begin{equation}
\int_{\Lambda_{\varepsilon}}\left\langle \frac{a(|\nabla u|)}{|\nabla
u|}\nabla u-\frac{a(|\nabla v|)}{|\nabla v|}\nabla v,\nabla u-\nabla
v\right\rangle dx\leq0. \label{integralnegativa}%
\end{equation}

On the other hand,%

\begin{gather*}
\left\langle \frac{a(|\nabla u|)}{|\nabla u|}\nabla u-\frac{a(|\nabla
v|)}{|\nabla v|}\nabla v,\nabla u-\nabla v\right\rangle \\
=a(|\nabla u|)|\nabla u|^{2}-\frac{a(|\nabla u|)}{|\nabla u|}\langle\nabla
u,\nabla v\rangle-\frac{a(|\nabla v|)}{|\nabla v|}\langle\nabla u,\nabla
v\rangle+a(|\nabla v|)\left\vert \nabla v\right\vert \\
\geq a(|\nabla u|)|\nabla u|-a(|\nabla u|)|\nabla v|-a(|\nabla v|)|\nabla
u|+a(|\nabla v|)|\nabla v|\\
=\left(  a(|\nabla u|)-a(|\nabla v|)\right)  \left(  |\nabla u|-|\nabla
v|\right)  ,
\end{gather*}
where the inequality is implied by Cauchy-Schwarz inequality. Since $a$ is
increasing it follows from (\ref{integralnegativa}) that $\left\vert \nabla
u\right\vert =\left\vert \nabla v\right\vert $ a.e. on $\Lambda_{\varepsilon}%
$. From this, in connection with (\ref{integralnegativa}), we conclude that
$\nabla\zeta=0$ a.e. on $\Omega.$ It follows then, again from
(\ref{integralnegativa}), that $\zeta=0$ since $\zeta\in C_{0}^{0,1}(\Omega)$.
We conclude that $u-v\leq0$ in $\Omega$ since $\varepsilon>0$ is arbitrary,
concluding with the proof of the proposition.
\end{proof}

\begin{proposition}
[Maximum Principle]\label{maxp} Let $\Omega\subset\mathcal{M}$ be an open
bounded and $u,v\in C^{0,1}(\Omega)\cap C^{0}\left(  \overline{\Omega}\right)
$ be weak solutions of $Q.$ Then%
\begin{equation}
\max_{\Omega}\left\vert u-v\right\vert =\max_{\partial\Omega}\left\vert
u-v\right\vert . \label{mp}%
\end{equation}
In particular, since $v=0$ is a solution we have the maximum principle%
\[
\max_{\Omega}\left\vert u\right\vert =\max_{\partial\Omega}\left\vert
u\right\vert .
\]

\end{proposition}

\begin{proof}
Set
\[
M:=\max_{\partial\Omega}\left\vert u-v\right\vert .
\]
Then, from the Comparison Principle%
\[
u=u-v+v\leq v+M.
\]
Reversing the roles of $u$ and $v$ we get $v\leq u+M$ from which we get
(\ref{mp}).
\end{proof}

\newpage

\subsection{\label{bge}Boundary gradient estimates. Barriers.}

\qquad In this section we assume that $\Omega$ is a bounded domain of $C^{2}$
class in $\mathcal{M}$. We follow Serrin's treatment with some simplifications
\cite{S}. The cases of mild and strong decay of the eigenvalue ratio have to
be treated separately. Nevertheless the type of barrier that we use, defined
in what follows, will be the same in both cases.

We fix a number $\delta_{0}>0$ such that the function%
\[
d(x)=\operatorname*{distance}(x,\partial\Omega)
\]
is of $C^{2}$ class on the boundary strip%
\[
\overline{\Omega}_{\delta_{0}}=\left\{  x\in\overline{\Omega}\text{
$\vert$
}0\leq d(x)\leq\delta_{0}\right\}
\]
and we seek barriers of the form%
\[
w=g+f(d)
\]
where the function $f$ is defined in some interval $\left[  0,\delta\right]
,$ $0<\delta\leq\delta_{0},$ $f(0)=0,$ and $w$ is a supersolution with
$f(\delta)=M$ or a subsolution with $f(\delta)=-M,$ $M>0$ a preassigned number.

To simplify later calculations we introduce the linear operator%
\[
\mathcal{L}_{w}v=\Delta v+b\left(  \left\vert \nabla w\right\vert \right)
\nabla^{2}v\left(  \frac{\nabla w}{\left\vert \nabla w\right\vert }%
,\frac{\nabla w}{\left\vert \nabla w\right\vert }\right)
\]
which satisfies
\begin{equation}
\left\vert \mathcal{L}_{w}v\right\vert \leq\left(  n-1+1+b\right)  \left\vert
\nabla^{2}v\right\vert \leq nB\left\vert \nabla^{2}v\right\vert \label{estele}%
\end{equation}
with $B=\max\left\{  1,1+b\right\}  .$ Using (\ref{PDE2}), (\ref{be}) and the
equality
\begin{equation}
\nabla^{2}u\left(  \nabla u,\nabla u\right)  =\frac{1}{2}\left\langle
\nabla\left\vert \nabla u\right\vert ^{2},\nabla u\right\rangle \label{hesu}%
\end{equation}
we may write the operator $Q$ as%
\begin{equation}
Q\left[  w\right]  =\mathcal{L}_{w}g+f^{\prime}\mathcal{L}_{w}d+f^{\prime
\prime}\left(  1+b\left\langle \nabla d,\frac{\nabla w}{\left\vert \nabla
w\right\vert }\right\rangle ^{2}\right)  . \label{ele}%
\end{equation}

We shall also need the following obvious estimates of $\left\vert \nabla
w\right\vert $ in terms of $f^{\prime}$:%
\begin{equation}
f^{\prime}-c_{1}\leq\left\vert \nabla w\right\vert \leq f^{\prime}+c_{1}
\label{c1}%
\end{equation}
where $c_{1}=\max_{\Omega_{\delta_{0}}}\left\vert \nabla g\right\vert $ and
hence%
\begin{equation}
\frac{2}{3}f^{\prime}\leq\left\vert \nabla w\right\vert \leq\frac{4}%
{3}f^{\prime} \label{doister}%
\end{equation}
provided that
\begin{equation}
f^{\prime}\geq\alpha\geq\max\left\{  1,3c_{1}\right\}  \label{alf}%
\end{equation}
where the number $\alpha$ will be appropriately chosen later on. We construct
only a supersolution and assume then that $f^{\prime\prime}\leq0.$

\subsubsection{The case of mild decay of the eigenvalue ratio}

\qquad As mentioned at the introduction, this class falls into the Serrin's
category of \textquotedblleft regularly elliptic\textquotedblright\ equations
(see \cite{S}). We have%
\begin{gather}
1+b\left\langle \nabla d,\frac{\nabla w}{\left\vert \nabla w\right\vert
}\right\rangle ^{2}\geq\left(  1+b\right)  \left\langle \nabla d,\frac{\nabla
w}{\left\vert \nabla w\right\vert }\right\rangle ^{2}\nonumber\\
=\frac{1+b}{f^{\prime2}}\left\langle \nabla w-\nabla g,\frac{\nabla
w}{\left\vert \nabla w\right\vert }\right\rangle ^{2}\geq\frac{1+b}%
{f^{\prime2}}\left(  \left\vert \nabla w\right\vert -c_{1}\right)
^{2}\label{set}\\
\geq\frac{1}{4}\frac{1+b}{f^{\prime2}}\left\vert \nabla w\right\vert
^{2}\nonumber
\end{gather}
on account of (\ref{doister}) and (\ref{alf}). Inserting into (\ref{ele}),
recalling that $f^{\prime\prime}\leq0$ and observing (\ref{estele}) we obtain%
\[
\frac{4}{f^{\prime}B}Q\left[  w\right]  \leq4n\left(  \left\vert \nabla
^{2}g\right\vert +\left\vert \nabla^{2}d\right\vert \right)  +\frac
{f^{\prime\prime}}{f^{\prime3}}\frac{b+1}{B}\left\vert \nabla w\right\vert
^{2}.
\]

Setting
\[
C=4n\max_{\overline{\Omega}_{\delta}}\left(  \left\vert \nabla^{2}g\right\vert
+\left\vert \nabla^{2}d\right\vert \right)
\]
and observing that
\[
\frac{b+1}{B}=1+b^{-}%
\]
we obtain from (\ref{fi}) and (\ref{condp})%
\begin{equation}
\frac{4}{f^{\prime}B}Q\left[  w\right]  \leq C+\frac{f^{\prime\prime}%
}{f^{\prime3}}\varphi\left(  \frac{2}{3}f^{\prime}\right)  . \label{prebar}%
\end{equation}
in the last step having used (\ref{doister}) and the fact that $\varphi$ is
non-decreasing. Our task will be complete if we can find a solution to the
ordinary differential equation%
\begin{equation}
f^{\prime\prime}+C\frac{f^{\prime3}}{\varphi\left(  \frac{2}{3}f^{\prime
}\right)  }=0 \label{ode}%
\end{equation}
which is defined in some interval $\left[  0,\delta,\right]  $ to be
explicitly determined later, $0\leq\delta\leq\delta_{0},$ and satisfies
$f(0)=0,$ $f(\delta)=M$ and $f^{\prime}\left(  \delta\right)  =\alpha$ where
$M$ is a given positive number and now $\alpha$ is chosen as
\begin{equation}
\alpha=\max\left\{  \frac{M}{\delta_{0}},1,3c_{1}\right\}  . \label{alpha}%
\end{equation}

We rewrite (\ref{ode}) as an equation for the inverse function of $f^{\prime
},$ denoted by $h,$ that is%
\[
h^{\prime}(s)=-\frac{\varphi\left(  \frac{2}{3}s\right)  }{Cs^{3}},
\]
leading to%
\[
h(s)=\int_{s}^{\beta}\frac{\varphi\left(  \frac{2}{3}t\right)  }{Ct^{3}%
}dt,\text{ }\alpha\leq s\leq\beta,
\]
where $\beta$ is still to be determined.

The domain of the definition of $f$ is the interval $\left[  h(\beta
),h(\alpha)\right]  =\left[  0,\delta\right]  $ with%
\[
\delta=\int_{\alpha}^{\beta}\frac{\varphi\left(  \frac{2}{3}t\right)  }%
{Ct^{3}}dt
\]
and%
\begin{align*}
f(\delta)  &  =\int_{0}^{\delta}f^{\prime}(s)ds=-\int_{\beta}^{\alpha}%
t\frac{\varphi\left(  \frac{2}{3}t\right)  }{Ct^{3}}dt\\
&  =\int_{\alpha}^{\beta}\frac{\varphi\left(  \frac{2}{3}t\right)  }{Ct^{2}%
}dt.
\end{align*}

Due to (\ref{condp}) we may now choose $\beta$ and hence $\delta$ so that
$f(\delta)=M$. Moreover,
\[
\delta=\int_{\alpha}^{\beta}\frac{\varphi\left(  \frac{2}{3}t\right)  }%
{Ct^{3}}dt\leq\frac{1}{\alpha}\int_{\alpha}^{\beta}\frac{\varphi\left(
\frac{2}{3}t\right)  }{Ct^{2}}dt=\frac{M}{\alpha}\leq\delta_{0}%
\]
where we used that $\alpha\geq1$ for the first inequality and that $\alpha\geq
M/\delta_{0}$ for the second one. Replacing $f$ by $-f$ we obtain a
subsolution. This completes the construction of barriers for class of partial
differential equations with mild decay of the eigenvalue ratio.

\subsubsection{The case of strong decay of the eingenvalue ratio}

\qquad In this case it becomes necessary to restrict the geometry of
$\partial\Omega;$ we require that the mean curvature of $\partial\Omega$ as
well as of the level hypersurfaces of $d,$ $0\leq d\leq\delta_{0},$ is
nonnegative with respect to the normal vector $\nabla d.$ This is equivalent
to the condition%
\begin{equation}
\Delta d\leq0\text{ in }\overline{\Omega}_{\delta_{0}}. \label{mconv}%
\end{equation}

Since $\nabla^{2}d\left(  \nabla d,\nabla d\right)  =0$ we then obtain%
\begin{align*}
f^{\prime}\mathcal{L}_{w}d  &  \leq f^{\prime}b\nabla^{2}d\left(  \frac{\nabla
w}{\left\vert \nabla w\right\vert },\frac{\nabla w}{\left\vert \nabla
w\right\vert }\right) \\
&  \leq B\frac{f^{\prime2}}{\left\vert \nabla w\right\vert ^{2}}\left\vert
2\nabla^{2}d\left(  \nabla d,\nabla g\right)  +\frac{1}{f^{\prime}}\nabla
^{2}d\left(  \nabla g,\nabla g\right)  \right\vert \\
&  \leq Bc_{0}%
\end{align*}
where the constant $c_{0}$ only depends on
\[
\max_{\overline{\Omega}_{\delta_{0}}}\left(  \left\vert \nabla g\right\vert
+\left\vert \nabla^{2}d\right\vert \right)
\]
and (\ref{alf}) is assumed to hold further on. Inserting this last estimate
and (\ref{set}) in (\ref{ele}) we arrive at
\begin{align*}
\frac{4}{B}Q\left[  w\right]   &  \leq C+\frac{f^{\prime\prime}}{f^{\prime2}%
}\frac{1+b}{B}\left\vert \nabla w\right\vert ^{2}\\
&  \leq C+\frac{f^{\prime\prime}}{f^{\prime2}}\varphi\left(  \left\vert \nabla
w\right\vert \right) \\
&  \leq C+\frac{f^{\prime\prime}}{f^{\prime2}}\varphi\left(  \frac{4}%
{3}f^{\prime}\right)
\end{align*}
where the constant $C$ depends only on
\[
\max_{\overline{\Omega}_{\delta_{0}}}\left(  \left\vert \nabla g\right\vert
+\left\vert \nabla^{2}g\right\vert +\left\vert \nabla^{2}d\right\vert \right)
\]
and we used (\ref{doister}), (\ref{condm}) and the fact that $\varphi$ is
non-increasing. We again choose $\alpha$ according to (\ref{alpha}) and
consider $h,$ the inverse function of $f^{\prime},$ which is given by
\[
h(s)=C\int_{s}^{\beta}\frac{\varphi\left(  \frac{4t}{3}\right)  }{t^{2}%
}dt,\text{ }\alpha\leq s\leq\beta.
\]

We obtain%
\[
\delta=C\int_{\alpha}^{\beta}\frac{\varphi\left(  \frac{4t}{3}\right)  }%
{t^{2}}dt
\]
and
\[
f\left(  \delta\right)  =C\int_{\alpha}^{\beta}\frac{\varphi\left(  \frac
{4t}{3}\right)  }{t}dt.
\]

Condition \ref{condm} allows to choose $\beta$ such that $f(\delta)=M$ and, as
before,
\[
\delta\leq\frac{M}{\alpha}\leq\delta_{0}%
\]
and the barrier construction for the class of minimal surface equation is complete.

\bigskip

From the previous calculations and also (\ref{c0apri}), we obtain:

\begin{theorem}
\label{bgem}Let $\Omega$ be a bounded domain of class $C^{2}$ in $\mathcal{M}$
and $u\in C^{1}\left(  \overline{\Omega}\right)  $ be a weak solution of
(\ref{PDE1}) such that $u=g$ on $\partial\Omega$ with $g\in C^{2}\left(
\overline{\Omega}\right)  .$ We assume that either Condition I or II of
Section \ref{intr} are satisfied and in case that Condition II holds we
require furthermore that the mean curvature of $\partial\Omega$ with respect
to the interior normal of $\partial\Omega$ as well as of the inner parallel
hypersurfaces of $\partial\Omega$ in some neighborhood of $\partial\Omega$ is
non negative. Then the normal derivative of $u$ on $\partial\Omega$ can be
estimated by a constant depending only on $\left\vert g\right\vert
_{C_{2}\left(  \Omega\right)  }$.
\end{theorem}

\newpage

\subsection{\label{glge}Global and local gradient estimates}

\qquad In this section we prove global and local estimates of solutions of the
partial differential equation (\ref{PDE1}) on bounded domains. We assume that
$u$ is a solution of class $C^{3}$ and use the equivalent form of
(\ref{PDE1}), namely:%
\begin{equation}
|\nabla u|^{2}\Delta u+b\nabla^{2}u\left(  \nabla u,\nabla u\right)  =0
\label{pde3rep}%
\end{equation}
recalling that
\[
b(s)=\frac{sa^{\prime}(s)}{a(s)}-1.
\]
\bigskip

In order to derive gradient bounds for the solutions of (\ref{pde3rep}) we
consider a point of $\Omega,$ say $x_{0},$ where a certain auxiliary function
attains a local maximum. We need slightly different such auxiliary functions,
all of them of the form%
\[
G\left(  x\right)  =g(x)f(u)F(|\nabla u|).
\]
The gradient estimates (local and global) are obtained by writing
\[
\nabla^{2}G(x_{0})\left(  \nabla u,\nabla u\right)
\]
as a polynomial in $\left\vert \nabla u\right\vert $. Analyzing its leading
coefficient, after an appropriate choice of $g,$ $f$ and $F,$ the constraint%
\[
\nabla^{2}G(x_{0})\left(  \nabla u,\nabla u\right)  \leq0
\]
will impose an upper bound for $\left\vert \nabla u\right\vert .$

We shall make use of the well known Bochner formula:

\begin{proposition}
[Bochner formula]\label{boch}If $\mathcal{M}^{n}$ is a Riemannian manifold and
$u\in C^{3}(\mathcal{M})$ then%
\begin{equation}
\left\langle \nabla\Delta u,\nabla u\right\rangle =\frac{1}{2}\Delta\left\vert
\nabla u\right\vert ^{2}-\left\vert \nabla^{2}u\right\vert ^{2}%
-\operatorname*{Ric}\left(  \nabla u,\nabla u\right)  \label{bochf}%
\end{equation}

\end{proposition}

\begin{proof}
Let $p\in M$ and $E_{1},...,E_{n}$ a local orthonormal frame field in a
neighborhood $V$ of $p$ such that
\begin{equation}
\nabla_{E_{i}}E_{j}(p)=0,~i,j=1,...,n. \label{bo1}%
\end{equation}
Hence, we have at $p$%
\begin{align}
\Delta\left\vert \nabla u\right\vert ^{2}  &  =\sum_{j}\nabla^{2}\left\vert
\nabla u\right\vert ^{2}\left(  E_{j},E_{j}\right) \nonumber\\
&  =\sum_{j}\left\langle \nabla_{E_{j}}\nabla\left\vert \nabla u\right\vert
^{2},E_{j}\right\rangle \nonumber\\
&  =\sum_{j}E_{j}\left\langle \nabla|\nabla u|^{2},E_{j}\right\rangle .
\label{bo2}%
\end{align}

Since%
\begin{align}
\left\langle \nabla|\nabla u|^{2},E_{j}\right\rangle  &  =E_{j}|\nabla
u|^{2}=E_{j}\left\langle \nabla u,\nabla u\right\rangle \nonumber\\
&  =2\left\langle \nabla_{E_{j}}\nabla u,\nabla u\right\rangle =2\nabla
^{2}u\left(  E_{j},\nabla u\right) \nonumber\\
&  =2\nabla^{2}u\left(  E_{j},\sum_{i}\left\langle \nabla u,E_{i}\right\rangle
E_{i}\right) \nonumber\\
&  =2\sum_{i}\left\langle \nabla u,E_{i}\right\rangle \nabla^{2}u\left(
E_{j},E_{i}\right)  \label{bot}%
\end{align}
hold at every point of $M$ we obtain%
\begin{align*}
\Delta\left\vert \nabla u\right\vert ^{2}  &  =2\sum_{i,j}E_{j}\left(
\left\langle \nabla u,E_{i}\right\rangle \nabla^{2}u\left(  E_{j}%
,E_{i}\right)  \right) \\
&  =2\sum_{i,j}\left[  E_{j}\left(  \left\langle \nabla u,E_{i}\right\rangle
\right)  \nabla^{2}u\left(  E_{j},E_{i}\right)  +\left\langle \nabla
u,E_{i}\right\rangle E_{j}\left(  \nabla^{2}u\left(  E_{j},E_{i}\right)
\right)  \right]
\end{align*}
and then
\begin{equation}
\Delta\left\vert \nabla u\right\vert ^{2}=2\sum_{i,j}\left[  \nabla
^{2}u\left(  E_{j},E_{i}\right)  ^{2}+\left\langle \nabla u,E_{i}\right\rangle
E_{j}\left(  \nabla^{2}u\left(  E_{j},E_{i}\right)  \right)  \right]  .
\label{bo3}%
\end{equation}

By the symmetry of $\nabla^{2}u$ we have at $p$%
\begin{align}
E_{j}\left(  \nabla^{2}u\left(  E_{j},E_{i}\right)  \right)   &  =E_{j}\left(
\nabla^{2}u\left(  E_{i},E_{j}\right)  \right) \nonumber\\
&  =E_{j}\left\langle \nabla_{E_{i}}\nabla u,E_{j}\right\rangle \nonumber\\
&  =\left\langle \nabla_{E_{j}}\nabla_{E_{i}}\nabla u,E_{j}\right\rangle
\nonumber\\
&  =\left\langle R\left(  E_{j},E_{i}\right)  \nabla u+\nabla_{E_{i}}%
\nabla_{E_{j}}\nabla u,E_{j}\right\rangle \nonumber\\
&  =\left\langle R\left(  E_{j},E_{i}\right)  \nabla u,E_{j}\right\rangle
+E_{i}\left\langle \nabla_{E_{j}}\nabla u,E_{j}\right\rangle \nonumber\\
&  =\left\langle R\left(  E_{j},E_{i}\right)  \nabla u,E_{j}\right\rangle
+E_{i}\left(  \nabla^{2}u\left(  E_{j},E_{j}\right)  \right)  \label{bo4}%
\end{align}
where $R$ denotes the curvature tensor of $\mathcal{M}.$ Inserting (\ref{bo4})
in (\ref{bo3}) we finally arrive at%
\begin{align*}
\frac{\Delta\left\vert \nabla u\right\vert ^{2}}{2}  &  =\sum_{i,j}\left\{
\nabla^{2}u\left(  E_{j},E_{i}\right)  ^{2}+\left\langle \nabla u,E_{i}%
\right\rangle \left[  \left\langle R\left(  E_{j},E_{i}\right)  \nabla
u,E_{j}\right\rangle +E_{i}\left(  \nabla^{2}u\left(  E_{j},E_{j}\right)
\right)  \right]  \right\} \\
&  =\left\vert \nabla^{2}u\right\vert ^{2}+\sum_{j}\left\langle R\left(
E_{j},\nabla u\right)  \nabla u,E_{j}\right\rangle +\sum_{i}\left\langle
\nabla u,E_{i}\right\rangle E_{i}\left(  \Delta u\right) \\
&  =\left\vert \nabla^{2}u\right\vert ^{2}+\operatorname{Ric}(\nabla u,\nabla
u)+\left\langle \nabla\Delta u,\nabla u\right\rangle .
\end{align*}

\end{proof}

We now obtain an equation for $|\nabla u|$ by differentiating (\ref{pde3rep})
in direction $\nabla u.$

\begin{lemma}
\label{pdegrad}If $u$ solves (\ref{pde3rep}) then, in an orthonormal frame
$E_{1},...,E_{n}$ with $E_{1}=\left\vert \nabla u\right\vert ^{-1}\nabla u$ on
a neighborhood of $\Omega$ where $\nabla u$ is non zero$,$ the following
equality holds%
\begin{gather*}
\left(  b+1\right)  \left\vert \nabla u\right\vert \nabla^{2}\left\vert \nabla
u\right\vert \left(  E_{1},E_{1}\right)  +\left\vert \nabla u\right\vert
\sum_{i=2}^{n}\nabla^{2}\left\vert \nabla u\right\vert \left(  E_{i}%
,E_{i}\right) \\
+b^{\prime}\left\vert \nabla u\right\vert \nabla^{2}u\left(  E_{1}%
,E_{1}\right)  ^{2}+b\sum_{i=2}^{n}\nabla^{2}u\left(  E_{1},E_{i}\right)
^{2}\\
-\sum_{i=1,j=2}^{n}\nabla^{2}u\left(  E_{i},E_{j}\right)  -\operatorname*{Ric}%
\left(  \nabla u,\nabla u\right)  =0,
\end{gather*}
where $\operatorname*{Ric}$ denotes the Ricci tensor of $\mathcal{M}.$
\end{lemma}

\begin{proof}
Differentiating (\ref{pde3rep}) in direction $\nabla u$ gives%
\begin{gather*}
\left\vert \nabla u\right\vert ^{2}\left\langle \nabla\Delta u,\nabla
u\right\rangle +\left\langle \nabla\left\vert \nabla u\right\vert ^{2},\nabla
u\right\rangle \Delta u\\
+b^{\prime}\frac{1}{2}\left\vert \nabla u\right\vert ^{-1}\left\langle
\nabla\left\vert \nabla u\right\vert ^{2},\nabla u\right\rangle \nabla
^{2}u\left(  \nabla u,\nabla u\right) \\
+b\nabla u\left(  \nabla^{2}u(\nabla u,\nabla u)\right)  =0.
\end{gather*}
From the relation
\[
\nabla^{2}u\left(  \nabla u,\nabla u\right)  =\frac{1}{2}\left\langle
\nabla\left\vert \nabla u\right\vert ^{2},\nabla u\right\rangle
\]
and Bochner formula (\ref{bochf}) we obtain%
\begin{gather*}
\frac{1}{2}\Delta\left\vert \nabla u\right\vert ^{2}-\left\vert \nabla
^{2}u\right\vert ^{2}-\operatorname*{Ric}\left(  \nabla u,\nabla u\right) \\
+\left(  b^{\prime}\left\vert \nabla u\right\vert -2b\right)  \left\vert
\nabla u\right\vert ^{-4}\nabla^{2}u\left(  \nabla u,\nabla u\right)  ^{2}\\
+\frac{1}{2}b\left\vert \nabla u\right\vert ^{-2}\nabla u\left\langle
\nabla\left\vert \nabla u\right\vert ^{2},\nabla u\right\rangle =0.
\end{gather*}

For the last term we have%
\begin{align*}
\nabla u\left\langle \nabla\left\vert \nabla u\right\vert ^{2},\nabla
u\right\rangle  &  =\left\vert \nabla u\right\vert \left[  \left\langle
\nabla_{E_{1}}\nabla\left\vert \nabla u\right\vert ^{2},\nabla u\right\rangle
+\left\langle \nabla\left\vert \nabla u\right\vert ^{2},\nabla_{E_{1}}\nabla
u\right\rangle \right] \\
&  =\left\vert \nabla u\right\vert \left[  \nabla^{2}\left\vert \nabla
u\right\vert ^{2}\left(  E_{1},\nabla u\right)  +\sum_{i=1}^{n}\left\langle
E_{i}\left(  \left\vert \nabla u\right\vert ^{2}\right)  E_{i},\nabla_{E_{1}%
}\nabla u\right\rangle \right] \\
&  =\left\vert \nabla u\right\vert ^{2}\left[  \nabla^{2}\left\vert \nabla
u\right\vert ^{2}\left(  E_{1},E_{1}\right)  +2\sum_{i=1}^{n}\nabla
^{2}u\left(  E_{1},E_{i}\right)  ^{2}\right]  .
\end{align*}
This leads to%
\begin{gather*}
\frac{1}{2}\left(  b+1\right)  \nabla^{2}\left\vert \nabla u\right\vert
^{2}\left(  E_{1},E_{1}\right)  +\frac{1}{2}\sum_{i=2}^{n}\nabla^{2}\left\vert
\nabla u\right\vert ^{2}\left(  E_{i},E_{i}\right) \\
+\left(  b^{\prime}\left\vert \nabla u\right\vert -b\right)  \nabla
^{2}u\left(  E_{1},E_{1}\right)  ^{2}+b\sum_{i=2}^{n}\nabla^{2}u\left(
E_{1},E_{i}\right)  ^{2}\\
-\sum_{i,j=1}^{n}\nabla^{2}u\left(  E_{i},E_{j}\right)  ^{2}%
-\operatorname*{Ric}\left(  \nabla u,\nabla u\right)  =0.
\end{gather*}

Using finally the relation%
\[
\frac{1}{2}\nabla^{2}\left\vert \nabla u\right\vert ^{2}\left(  E_{i}%
,E_{i}\right)  =\left\vert \nabla u\right\vert \nabla^{2}\left\vert \nabla
u\right\vert \left(  E_{i},E_{i}\right)  +\nabla^{2}u\left(  E_{1}%
,E_{i}\right)  ^{2},\text{ }1\leq i\leq n,
\]
to convert the last equation into one for $\left\vert \nabla u\right\vert $
instead of $\left\vert \nabla u\right\vert ^{2}$ we arrive at the equation in
the lemma.
\end{proof}

We resume some computations used in the estimates in the following lemma:

\begin{lemma}
\label{lem2} If $u$ solves (\ref{pde3rep}) and the function
$G(x)=g(x)f(u)F(|\nabla u|)$ attains a local maximum in an interior point
$y_{0}$ of $\Omega$ with $\nabla u(y_{0})\neq0$ then, in terms of a local
orthonormal basis $E_{1}:=|\nabla u|^{-1}\nabla u,E_{2},\ldots,E_{n}$ of
$T_{y_{0}}\mathcal{M}$ we obtain, at $y_{0},$ the relations
\[
\frac{F^{\prime}}{F}\nabla^{2}u(E_{1},E_{i})=-\frac{1}{g}\langle\nabla
g,E_{i}\rangle-\frac{f^{\prime}}{f}\langle\nabla u,E_{i}\rangle
\]
and
\begin{align*}
0  &  \geq\left[  -\frac{F^{\prime}b^{\prime}}{F}+(b+1)(\frac{F^{\prime\prime
}}{F}-\frac{F^{\prime2}}{F^{2}})\right]  \nabla^{2}u(E_{1},E_{1})^{2}%
+\frac{F^{\prime}}{F|\nabla u|}\sum_{\substack{i,j\\i\geq2}}\nabla^{2}%
u(E_{i},E_{j})^{2}\\
&  +\left[  -\frac{F^{\prime}b}{F|\nabla u|}+\frac{F^{\prime\prime}}{F}%
-\frac{F^{\prime2}}{F^{2}}\right]  \sum_{i\geq2}\nabla^{2}u(E_{1},E_{i}%
)^{2}+(b+1)(\frac{f^{\prime\prime}}{f}-\frac{f^{\prime2}}{f^{2}})|\nabla
u|^{2}\\
&  +\frac{|\nabla u|F^{\prime}}{F}\operatorname*{Ric}(E_{1},E_{1})+\frac{1}%
{g}\left[  (b+1)\nabla^{2}g(E_{1},E_{1})+\sum_{i\geq2}\nabla^{2}g(E_{i}%
,E_{i})\right] \\
&  -\frac{1}{g^{2}}\left[  (b+1)\langle\nabla g,E_{1}\rangle^{2}+\sum_{i\geq
2}\langle\nabla g,E_{i}\rangle^{2}\right]  .
\end{align*}

\end{lemma}

\begin{proof}
We compute
\begin{equation}
\nabla\ln G=\frac{1}{g}\nabla g+\frac{f^{\prime}}{f}\nabla u+\frac{F^{\prime}%
}{F}\nabla|\nabla u| \label{grG}%
\end{equation}
and, for $\xi,\eta\in T_{y_{0}}M,$%
\begin{align*}
\nabla^{2}\ln G(\xi,\eta)  &  =\langle\nabla_{\xi}\nabla\ln G,\eta
\rangle=\langle\nabla_{\xi}\left(  \frac{1}{g}\nabla g+\frac{f^{\prime}}%
{f}\nabla u+\frac{F^{\prime}}{F}\nabla|\nabla u|\right)  ,\eta\rangle\\
&  =\langle-\frac{1}{g^{2}}\xi(g)\nabla g+\frac{1}{g}\nabla_{\xi}\nabla
g+[\frac{f^{\prime\prime}}{f}\xi(u)-\frac{f^{\prime2}}{f^{2}}\xi(u)]\nabla
u+\frac{f^{\prime}}{f}\nabla_{\xi}\nabla u\\
&  +[\frac{F^{\prime\prime}}{F}\xi(|\nabla u|)-\frac{F^{\prime2}}{F^{2}}%
\xi(|\nabla u|)]\nabla|\nabla u|+\frac{F^{\prime}}{F}\nabla_{\xi}\nabla|\nabla
u|,\eta\rangle\\
&  =-\frac{1}{g^{2}}\langle\nabla g,\xi\rangle\langle\nabla g,\eta
\rangle+\frac{1}{g}\nabla^{2}g(\xi,\eta)+[\frac{f^{\prime\prime}}{f}%
-\frac{f^{\prime2}}{f^{2}}]\langle\nabla u,\xi\rangle\langle\nabla
u,\eta\rangle\\
&  +\frac{f^{\prime}}{f}\nabla^{2}u(\xi,\eta)+[\frac{F^{\prime\prime}}%
{F}-\frac{F^{\prime2}}{F^{2}}]\frac{1}{|\nabla u|^{2}}\nabla^{2}u(\nabla
u,\xi)\nabla^{2}u(\nabla u,\eta)\\
&  +\frac{F^{\prime}}{F}\nabla^{2}|\nabla u|(\xi,\eta)
\end{align*}
where we used the relation
\begin{align}
\nabla^{2}u(\nabla u,\eta)  &  =\langle\nabla_{\eta}\nabla u,\nabla
u\rangle=\frac{1}{2}\eta(|\nabla u|^{2})\label{re1}\\
&  =\frac{1}{2}\langle\nabla|\nabla u|^{2},\eta\rangle=|\nabla u|\langle
\nabla|\nabla u|,\eta\rangle.\nonumber
\end{align}

By (\ref{grG}) and (\ref{re1}) we have at $y_{0}$
\begin{align*}
\frac{F^{\prime}}{F}\nabla^{2}u(E_{1},E_{i})  &  =\frac{F^{\prime}}{F|\nabla
u|}\nabla^{2}u(\nabla u,E_{i})\\
&  =\frac{F^{\prime}}{F}\langle\nabla|\nabla u|,E_{i}\rangle=-\frac{1}%
{g}\langle\nabla g,E_{i}\rangle-\frac{f^{\prime}}{f}\langle\nabla
u,E_{i}\rangle
\end{align*}
Since (\ref{pde3rep}) is elliptic and the matrix $(\nabla^{2}\ln G_{y_{0}%
}(E_{i},E_{j}))$ is nonpositive, we obtain at $y_{0}$
\begin{align*}
0\geq\theta &  =(b+1)\nabla^{2}\ln G(E_{1},E_{1})+\sum_{i\geq2}\nabla^{2}\ln
G(E_{i},E_{i})\\
&  =-\frac{b+1}{g^{2}}\langle\nabla g,E_{1}\rangle^{2}+\frac{b+1}{g}\nabla
^{2}g(E_{1},E_{1})+(b+1)(\frac{f^{\prime\prime}}{f}-\frac{f^{\prime2}}{f^{2}%
})\langle\nabla u,E_{1}\rangle^{2}\\
&  +(b+1)\frac{f^{\prime}}{f}\nabla^{2}u(E_{1},E_{1})+(b+1)(\frac
{F^{\prime\prime}}{F}-\frac{F^{\prime2}}{F^{2}})\frac{1}{|\nabla u|^{2}}%
\nabla^{2}u(\nabla u,E_{1})^{2}\\
&  +(b+1)\frac{F^{\prime}}{F}\nabla^{2}|\nabla u|(E_{1},E_{1})-\frac{1}{g^{2}%
}\sum_{i\geq2}\left\langle \nabla g,E_{i}\right\rangle ^{2}+\frac{1}{g}%
\sum_{i\geq2}\nabla^{2}g(E_{i},E_{i})\\
&  +(\frac{f^{\prime\prime}}{f}-\frac{f^{\prime2}}{f^{2}})\sum_{i\geq2}%
\langle\nabla u,E_{i}\rangle^{2}+\frac{f^{\prime}}{f}\sum_{i\geq2}\nabla
^{2}u(E_{i},E_{i})\\
&  +(\frac{F^{\prime\prime}}{F}-\frac{F^{\prime2}}{F^{2}})\frac{1}{|\nabla
u|^{2}}\sum_{i\geq2}\nabla^{2}u(\nabla u,E_{i})^{2}+\frac{F^{\prime}}{F}%
\sum_{i\geq2}\nabla^{2}|\nabla u|(E_{i},E_{i})
\end{align*}
Hence, by (\ref{pde3rep}),
\begin{align}
\theta &  =\frac{F^{\prime}}{F|\nabla u|}\left\{  |\nabla u|(b+1)\nabla
^{2}|\nabla u|(E_{1},E_{1})+|\nabla u|\sum_{i\geq2}\nabla^{2}|\nabla u|\left(
E_{i},E_{i}\right)  \right\} \nonumber\\
&  +(b+1)(\frac{F^{\prime\prime}}{F}-\frac{F^{\prime2}}{F^{2}})\nabla
^{2}u(E_{1},E_{1})^{2}+(\frac{F^{\prime\prime}}{F}-\frac{F^{\prime2}}{F^{2}%
})\sum_{i\geq2}\nabla^{2}u(E_{1},E_{i})^{2}\nonumber\\
&  +(b+1)(\frac{f^{\prime\prime}}{f}-\frac{f^{\prime2}}{f^{2}})|\nabla
u|^{2}+\frac{1}{g}\left[  (b+1)\nabla^{2}g(E_{1},E_{1})+\sum_{i\geq2}%
\nabla^{2}g(E_{i},E_{i})\right] \nonumber\\
&  -\frac{1}{g^{2}}\left[  (b+1)\langle\nabla g,E_{1}\rangle^{2}+\sum_{i\geq
2}\langle\nabla g,E_{i}\rangle^{2}\right]  +\frac{f^{\prime}}{f}%
\underset{=\text{ }0}{\underbrace{\left[  b\nabla^{2}u(E_{1},E_{1})+\Delta
u\right]  }} \label{tk1}%
\end{align}
By Lemma \ref{pdegrad} and (\ref{tk1}), we obtain
\begin{align*}
\theta &  =\frac{F^{\prime}}{F|\nabla u|}\left\{  -b^{\prime}\left\vert \nabla
u\right\vert (E_{1},E_{1})^{2}-b\sum_{i\geq2}\nabla^{2}u(E_{1},E_{i})^{2}%
+\sum_{\substack{i,j\\i\geq2}}\nabla^{2}u(E_{i},E_{j})^{2}%
+\text{$\operatorname*{Ric}$}(\nabla u,\nabla u)\right\} \\
&  \;+(b+1)(\frac{F^{\prime\prime}}{F}-\frac{F^{\prime2}}{F^{2}})\nabla
^{2}u(E_{1},E_{1})^{2}+(\frac{F^{\prime\prime}}{F}-\frac{F^{\prime2}}{F^{2}%
})\sum_{i\geq2}\nabla^{2}u(E_{1},E_{i})\\
&  +(b+1)(\frac{f^{\prime\prime}}{f}-\frac{f^{\prime2}}{f^{2}})|\nabla
u|^{2}+\frac{1}{g}\left[  (b+1)\nabla^{2}g(E_{1},E_{1})+\sum_{i\geq2}%
\nabla^{2}g(E_{i},E_{i})\right] \\
&  -\frac{1}{g^{2}}\left[  (b+1)\langle\nabla g,E_{1}\rangle^{2}+\sum_{i\geq
2}\langle\nabla g,E_{i}\rangle^{2}\right]
\end{align*}
and so
\begin{align*}
\theta &  =\left[  -\frac{F^{\prime}b^{\prime}}{F}+(b+1)(\frac{F^{\prime
\prime}}{F}-\frac{F^{\prime2}}{F^{2}})\right]  \nabla^{2}u(E_{1},E_{1}%
)^{2}+\frac{F^{\prime}}{F|\nabla u|}\sum_{\substack{i,j\\i\geq2}}\nabla
^{2}u(E_{i},E_{j})^{2}\\
&  +\left[  -\frac{F^{\prime}b}{F|\nabla u|}+\frac{F^{\prime\prime}}{F}%
-\frac{F^{\prime2}}{F^{2}}\right]  \sum_{i\geq2}\nabla^{2}u(E_{1},E_{i}%
)^{2}+(b+1)(\frac{f^{\prime\prime}}{f}-\frac{f^{\prime2}}{f^{2}})|\nabla
u|^{2}\\
&  +\frac{|\nabla u|F^{\prime}}{F}\text{$\operatorname*{Ric}$}(E_{1}%
,E_{1})+\frac{1}{g}\left[  (b+1)\nabla^{2}g(E_{1},E_{1})+\sum_{i\geq2}%
\nabla^{2}g(E_{i},E_{i})\right] \\
&  -\frac{1}{g^{2}}\left[  (b+1)\langle\nabla g,E_{1}\rangle^{2}+\sum_{i\geq
2}\langle\nabla g,E_{i}\rangle^{2}\right]  .
\end{align*}

\end{proof}

\subsubsection{The class of mild decay of the eigenvalue ratio}

\qquad Taking the function $F(s)=s$ in Lemma \ref{lem2}, we obtain
\begin{equation}
\frac{\nabla^{2}u(E_{1},E_{1})^{2}}{|\nabla u|^{2}}=\frac{f^{\prime2}}{f^{2}%
}|\nabla u|^{2}+\frac{1}{g^{2}}\langle\nabla g,E_{1}\rangle^{2}+\frac
{2f^{\prime}}{fg}\langle\nabla g,E_{1}\rangle|\nabla u|, \label{tk3}%
\end{equation}%
\begin{equation}
\frac{1}{|\nabla u|^{2}}\nabla^{2}u(E_{1},E_{i})^{2}=\frac{1}{g^{2}}%
\langle\nabla g,E_{i}\rangle^{2},\;\;\forall i=2,\ldots,n. \label{tk4}%
\end{equation}
and
\begin{align}
0  &  \geq-\frac{(b^{\prime}|\nabla u|+b+1)}{|\nabla u|^{2}}\nabla^{2}%
u(E_{1},E_{1})^{2}+\frac{1}{|\nabla u|^{2}}\sum_{\substack{i,j\\i\geq2}%
}\nabla^{2}u(E_{i},E_{j})^{2}\nonumber\\
&  -\frac{b+1}{|\nabla u|^{2}}\sum_{i\geq2}\nabla^{2}u(E_{1},E_{i}%
)^{2}+(b+1)(\frac{f^{\prime\prime}}{f}-\frac{f^{\prime2}}{f^{2}})|\nabla
u|^{2}\nonumber\\
&  +\text{$\operatorname*{Ric}$}(E_{1},E_{1})+\frac{1}{g}\left[
(b+1)\nabla^{2}g(E_{1},E_{1})+\sum_{i\geq2}\nabla^{2}g(E_{i},E_{i})\right]
\nonumber\\
&  -\frac{1}{g^{2}}\left[  (b+1)\langle\nabla g,E_{1}\rangle^{2}+\sum_{i\geq
2}\langle\nabla g,E_{i}\rangle^{2}\right]  \label{tf1}%
\end{align}
Inserting (\ref{tk3}) and (\ref{tk4}) in (\ref{tf1}), we arrive at%
\begin{align}
0  &  \geq-(b^{\prime}|\nabla u|+b+1)\left[  \frac{f^{\prime2}}{f^{2}}|\nabla
u|^{2}+\frac{1}{g^{2}}\langle\nabla g,E_{1}\rangle^{2}+\frac{2f^{\prime}}%
{fg}\langle\nabla g,E_{1}\rangle|\nabla u|\right] \nonumber\\
&  -\left(  b+1\right)  \sum_{i\geq2}\frac{1}{g^{2}}\langle\nabla
g,E_{i}\rangle^{2}+(b+1)(\frac{f^{\prime\prime}}{f}-\frac{f^{\prime2}}{f^{2}%
})|\nabla u|^{2}\nonumber\\
&  +\text{$\operatorname*{Ric}$}(E_{1},E_{1})+\frac{1}{g}\left[
(b+1)\nabla^{2}g(E_{1},E_{1})+\sum_{i\geq2}\nabla^{2}g(E_{i},E_{i})\right]
\nonumber\\
&  -\frac{1}{g^{2}}\left[  (b+1)\langle\nabla g,E_{1}\rangle^{2}+\sum_{i\geq
2}\langle\nabla g,E_{i}\rangle^{2}\right] \nonumber\\
&  =-b^{\prime}|\nabla u|\frac{f^{\prime2}}{f^{2}}|\nabla u|^{2}+(b+1)\left(
\frac{f^{\prime\prime}}{f}-2\frac{f^{\prime2}}{f^{2}}\right)  |\nabla
u|^{2}+\text{$\operatorname*{Ric}$}(E_{1},E_{1})\nonumber\\
&  -(b^{\prime}|\nabla u|+b+1)\left[  \frac{\langle\nabla g,E_{1}\rangle^{2}%
}{g^{2}}+\frac{2f^{\prime}}{fg}\langle\nabla g,E_{1}\rangle|\nabla u|\right]
-\frac{b+1}{g^{2}}\sum_{i\geq2}\langle\nabla g,E_{i}\rangle^{2}\nonumber\\
&  +\frac{1}{g}\left[  (b+1)\nabla^{2}g(E_{1},E_{1})+\sum_{i\geq2}\nabla
^{2}g(E_{i},E_{i})\right] \nonumber\\
&  -\frac{1}{g^{2}}\left[  (b+1)\langle\nabla g,E_{1}\rangle^{2}+\sum_{i\geq
2}\langle\nabla g,E_{i}\rangle^{2}\right]  \label{eq5}%
\end{align}
We first consider the global estimate where we set $g\equiv1$ and choose
$f(u)=(\ln(K+u))^{-1}$ with a constant $K>0$. For convenience we also assume
that $u\geq0$. For this $f$ we have
\[
f^{\prime}=-(K+u)^{-1}(\ln(K+u))^{-2},\;f^{\prime\prime}=(K+u)^{-2}%
(\ln(K+u))^{-3}\left(  \ln(K+u)+2\right)
\]
and thus
\[
\frac{f^{\prime}}{f}=-\frac{1}{(K+u)\ln(K+u)},\;\;\frac{f^{\prime\prime}}%
{f}=\frac{1}{(K+u)^{2}\ln(K+u)}+\frac{2f^{\prime2}}{f^{2}}.
\]
Then (\ref{eq5}) becomes
\[
0\geq-b^{\prime}|\nabla u|\frac{f^{\prime2}}{f^{2}}|\nabla u|^{2}+(b+1)\left[
\frac{f^{\prime\prime}}{f}-2\frac{f^{\prime2}}{f^{2}}\right]  |\nabla
u|^{2}+\text{$\operatorname*{Ric}$}(E_{1},E_{1})
\]
and so
\begin{equation}
(K+u)^{-2}(\ln(K+u))^{-1}\left[  b+1-\frac{b^{\prime+}\left(  |\nabla
u|\right)  }{\ln(K+u)}\right]  |\nabla u|^{2}\leq|\operatorname*{Ric}%
\nolimits^{-}|, \label{qc}%
\end{equation}
where $b^{\prime+}=\max\{b^{\prime},0\}$ and
\[
\operatorname*{Ric}\nolimits^{-}:=\min_{\left\vert \eta\right\vert =1}%
\min\left\{  \operatorname*{Ric}\left(  \eta,\eta\right)  ,0\right\}  .
\]

We now require the condition:

\begin{condition}
\label{cond6} There are numbers $s_{0},$ $\beta>0$ and a function $\varphi\in
C^{0}\left(  \left[  0,\infty\right)  \right)  $ with $\lim_{s\rightarrow
+\infty}\varphi(s)=+\infty$ such that
\[
(b(s)+1-\beta b^{\prime+}(s)s)s^{2}\geq\varphi(s)
\]
for $s\geq s_{0}.$
\end{condition}

Then, since (\ref{qc}) holds at a point where the function $G=|\nabla
u|/\ln(K+u)$ attains a maximum, choosing $K=\exp(1/\beta),$ we obtain

\begin{theorem}
\label{globplapl}Under Condition \ref{cond6} there is a constant $C$ depending
only on $\varphi$, $\beta$ and $\sup_{\Omega}\left(  \left\vert u\right\vert
+\operatorname*{Ric}\nolimits^{-}\right)  $ such\textit{ }that if the function
$|\nabla u|/\ln(K+u)$ attains a local maximum at an interior point $y_{0}$ of
$\Omega$, then
\[
\left\vert \nabla u(y_{0})\right\vert \leq C.
\]

\end{theorem}

We now turn our attention to the local estimates. Here we consider the
solution $u$ in a closed geodesic ball $B_{r}(x_{0})$ with center $x_{0}$ and
radius $r$ smaller than the distance of $x_{0}$ to its cut locus, if the
latter is nonempty, and we choose
\[
g(x)=1-\frac{\rho^{2}}{r^{2}},\;\rho(x)=\operatorname*{dist}(x,x_{0}).
\]
Unless $\nabla u\equiv0$ in $B_{r}(x_{0})$, what would make any further
estimate superfluous the function $\ln G$ attains a local maximum in some
point $y_{0}$ in the interior of $B_{r}(x_{0})$. In case that
\begin{equation}
g(y_{0})|\nabla u(y_{0})|\leq\frac{4}{r}\frac{f(u(y_{0}))}{|f^{\prime}%
(u(y_{0}))|} \label{eq8}%
\end{equation}
would hold, then the estimate of $|\nabla u(y_{0})|$ will turn out to be
trivial. Hence we shall assume that
\[
\frac{1}{g}\leq\frac{r|f^{\prime}|}{4f}|\nabla u|\;\;\text{at}\;y_{0}%
\]
what implies
\begin{align}
2\left\vert \frac{f^{\prime}}{fg}\langle\nabla g,E_{1}\rangle|\nabla
u|\right\vert  &  \leq2\frac{|f^{\prime}|}{fg}|\nabla g||\nabla u|\nonumber\\
&  \leq2\frac{|f^{\prime}||\nabla u|}{f}\frac{r|f^{\prime}||\nabla u|}%
{4f}\frac{|\nabla\rho^{2}|}{r^{2}}\nonumber\\
&  =\frac{f^{\prime2}|\nabla\rho^{2}||\nabla u|^{2}}{2f^{2}r}\nonumber\\
&  =\frac{f^{\prime2}\rho|\nabla\rho||\nabla u|^{2}}{f^{2}r}\nonumber\\
&  \leq\frac{f^{\prime2}|\nabla u|^{2}}{f^{2}} \label{eq9}%
\end{align}
Now we find it necessary to require a much stronger condition than Condition
\ref{cond6}, namely:

\begin{condition}
\label{cond10} there exist positive numbers $\alpha$, $\beta$ and $s_{0}$ such
that
\[
B(s)^{-1}(b(s)+1-\beta|b^{\prime}(s)|s)\geq\alpha,\;\forall s\geq s_{0},
\]
where, as before, $B(s)=\max\{1,1+b(s)\}$.
\end{condition}

It is immediate to see that Condition \ref{cond10} implies Conditions I and
\ref{cond6}.

From (\ref{eq5}) we have%
\begin{align*}
0  &  \geq-|b^{\prime}||\nabla u|\frac{f^{\prime2}}{f^{2}}|\nabla
u|^{2}+(b+1)\left(  \frac{f^{\prime\prime}}{f}-2\frac{f^{\prime2}}{f^{2}%
}\right)  |\nabla u|^{2}+\text{$\operatorname*{Ric}$}(E_{1},E_{1})\\
&  -(|b^{\prime}||\nabla u|+b+1)\left[  \frac{|\nabla g|^{2}}{g^{2}%
}+\left\vert \frac{2f^{\prime}}{fg}\langle\nabla g,E_{1}\rangle|\nabla
u|\right\vert \right]  -\frac{b+1}{g^{2}}|\nabla g|^{2}\\
&  -\frac{1}{g}\left[  (b+1)|\nabla^{2}g|+\sqrt{n-1}|\nabla^{2}g|\right]
-\frac{1}{g^{2}}\left[  (b+1)|\nabla g|^{2}+|\nabla g|^{2}\right]
\end{align*}

By means of (\ref{eq9}) one has
\begin{align*}
0  &  \geq-|b^{\prime}||\nabla u|\frac{f^{\prime2}}{f^{2}}|\nabla
u|^{2}+(b+1)\left(  \frac{f^{\prime\prime}}{f}-2\frac{f^{\prime2}}{f^{2}%
}\right)  |\nabla u|^{2}+\text{$\operatorname*{Ric}$}(E_{1},E_{1})\\
&  -(|b^{\prime}||\nabla u|+b+1)\left[  \frac{4}{g^{2}r^{2}}+\frac{f^{\prime
2}|\nabla u|^{2}}{f^{2}}\right]  -\frac{4(b+1)}{g^{2}r^{2}}\\
&  -\frac{1}{g}\left[  (b+1)\frac{|\nabla^{2}\rho^{2}|}{r^{2}}+\frac
{\sqrt{n-1}|\nabla^{2}\rho^{2}|}{r^{2}}\right]  -\frac{1}{g^{2}}\left[
(b+1)\frac{4}{r^{2}}+\frac{4}{r^{2}}\right] \\
&  =-|b^{\prime}||\nabla u|\frac{f^{\prime2}}{f^{2}}|\nabla u|^{2}%
+(b+1)\left(  \frac{f^{\prime\prime}}{f}-2\frac{f^{\prime2}}{f^{2}}\right)
|\nabla u|^{2}+\text{$\operatorname*{Ric}$}(E_{1},E_{1})\\
&  -(|b^{\prime}||\nabla u|+b+1)\frac{4}{g^{2}r^{2}}-\frac{|b^{\prime}||\nabla
u|f^{\prime2}|\nabla u|^{2}}{f^{2}}-(b+1)\frac{f^{\prime2}|\nabla u|^{2}%
}{f^{2}}\\
&  -\frac{4(b+1)}{g^{2}r^{2}}-\frac{1}{g}\left[  (b+1)\frac{|\nabla^{2}%
\rho^{2}|}{r^{2}}+\frac{\sqrt{n-1}|\nabla^{2}\rho^{2}|}{r^{2}}\right]
-\frac{1}{g^{2}}\left[  (b+1)\frac{4}{r^{2}}+\frac{4}{r^{2}}\right]
\end{align*}
and so
\begin{align*}
0  &  \geq-2|b^{\prime}||\nabla u|\frac{f^{\prime2}}{f^{2}}|\nabla
u|^{2}+(b+1)\left(  \frac{f^{\prime\prime}}{f}-3\frac{f^{\prime2}}{f^{2}%
}\right)  |\nabla u|^{2}+\text{$\operatorname*{Ric}$}(E_{1},E_{1})\\
&  -(|b^{\prime}||\nabla u|+b+1)\frac{4}{g^{2}r^{2}}-\frac{b+1+\sqrt{n-1}%
}{gr^{2}}|\nabla^{2}\rho^{2}|-\frac{1}{g^{2}}\left[  (b+1)\frac{8}{r^{2}%
}+\frac{4}{r^{2}}\right] \\
&  \geq-2|b^{\prime}||\nabla u|\frac{f^{\prime2}}{f^{2}}|\nabla u|^{2}%
+(b+1)\left(  \frac{f^{\prime\prime}}{f}-3\frac{f^{\prime2}}{f^{2}}\right)
|\nabla u|^{2}+\text{$\operatorname*{Ric}$}(E_{1},E_{1})\\
&  -(|b^{\prime}||\nabla u|+b+1)\frac{4}{g^{2}r^{2}}-\frac{\left(
1+\sqrt{n-1}\right)  B}{gr^{2}}|\nabla^{2}\rho^{2}|-\frac{12B}{g^{2}r^{2}}.
\end{align*}
Our Condition \ref{cond10} implies
\[
|b^{\prime}|s\leq\frac{1}{\beta}(b+1)
\]
and hence
\[
b+1+|b^{\prime}|s\leq(1+\frac{1}{\beta})(b+1)\leq(1+\frac{1}{\beta})B.
\]

Thus
\begin{align*}
&  (b+1)\left(  \frac{f^{\prime\prime}}{f}-3\frac{f^{\prime2}}{f^{2}}\right)
|\nabla u|^{2}-2|b^{\prime}||\nabla u|\frac{f^{\prime2}}{f^{2}}|\nabla
u|^{2}\\
&  \leq\frac{B}{g^{2}r^{2}}\left(  14+\frac{4}{\beta}+\left(  2+\sqrt
{n-1}\right)  |\nabla^{2}\rho^{2}|\right)  +|\operatorname*{Ric}\nolimits^{-}|
\end{align*}
Hence there is a constant $C_{0}=C_{0}(n,\beta)$ such that
\begin{align*}
&  B\left[  (b+1)\left(  \frac{f^{\prime\prime}}{f}-3\frac{f^{\prime2}}{f^{2}%
}\right)  |\nabla u|^{2}-2|b^{\prime}||\nabla u|\frac{f^{\prime2}}{f^{2}%
}|\nabla u|^{2}\right] \\
&  \leq\frac{C_{0}}{r^{2}g^{2}}\left(  1+|\nabla^{2}\rho^{2}|\right)
+|\operatorname*{Ric}\nolimits^{-}|.
\end{align*}
We again choose $f(u)=1/\ln(K+u)$. From the assumption that $u\geq0$, one
has\newline%

\begin{align*}
&  {(b+1)}\left(  {\frac{f^{\prime\prime}}{f}-\frac{3f^{\prime2}}{f^{2}}%
}\right)  {-\frac{2|b^{\prime}|sf^{\prime2}}{f^{2}}}\\
&  =\frac{1}{(K+u)^{2}\ln(K+u)}\left[  (1-\frac{1}{\ln(K+u)})(b+1)-\frac
{2|b^{\prime}|s}{\ln(K+u)}\right] \\
&  =\frac{(1-\frac{1}{\ln(K+u)})}{(K+u)^{2}\ln(K+u)}\left[  b+1-\frac
{2|b^{\prime}|s}{\ln(K+u)(1-\frac{1}{\ln(K+u)})}\right] \\
&  \geq\frac{{(1-\frac{1}{\ln K})}}{{(K+u)^{2}\ln(K+u)}}\left[  b+1-\frac
{2|b^{\prime}|s}{\ln K-1}\right]
\end{align*}
By (\ref{cond10}) we choose $K$ with $2/\left(  \ln K-1\right)  =\beta$ so
that
\[
(b+1)\left(  \frac{f^{\prime\prime}}{f}-\frac{3f^{\prime2}}{f^{2}}\right)
-\frac{2|b^{\prime}|sf^{\prime2}}{f^{2}}\geq\frac{\alpha B(1-(\ln K)^{-1}%
)}{(K+u)^{2}\ln(K+u)},\;\forall s\geq s_{0}%
\]
Therefore we get at $y_{0}$,
\[
g^{2}|\nabla u|^{2}\leq\frac{1}{\alpha}\left(  \frac{C_{0}}{r^{2}}\left(
1+|\nabla^{2}\rho^{2}|\right)  +|\text{$\operatorname*{Ric}$}|\right)
\frac{(K+u)^{2}\ln(K+u)}{(1-(\ln K)^{-1})},
\]
provided that $|\nabla u(y_{0})|\geq s_{0}$ and (\ref{eq8}) does not hold.
Since by construction $G(x_{0})\leq G(y_{0}),$ setting
\[
{M=\max_{B_{r}(x_{0})}u,}%
\]
we therefore either have
\begin{gather*}
\frac{|\nabla u(x_{0})|}{\ln(K+u(x_{0}))}\leq\frac{g(y_{0})|\nabla u(y_{0}%
)|}{\ln(K+u(y_{0}))}\\
\leq\left\{  \frac{1}{\alpha}\left[  \frac{C_{0}}{r^{2}}\left(  1+|\nabla
^{2}\rho^{2}|\right)  +|\operatorname*{Ric}\nolimits^{-}|\right]  \right\}
^{1/2}\frac{(K+u(y_{0}))(\ln(K+u(y_{0}))^{1/2}}{\ln(K+u(y_{0}))(1-(\ln
K)^{-1})^{1/2}}\\
\leq\left\{  \frac{1}{\alpha}\left[  \frac{C_{0}}{r^{2}}\left(  1+|\nabla
^{2}\rho^{2}|\right)  +|\operatorname*{Ric}\nolimits^{-}|\right]  \right\}
^{1/2}\frac{(K+M)}{(\ln K)^{1/2}(1-(\ln K)^{-1})^{1/2}}%
\end{gather*}
and so
\begin{align*}
&  \left\vert \nabla u(x_{0})\right\vert \\
&  \leq\left\{  \frac{1}{\alpha}\left[  \frac{C_{0}}{r^{2}}\left(
1+\max|\nabla^{2}\rho^{2}|\right)  +|\operatorname*{Ric}\nolimits^{-}|\right]
\right\}  ^{1/2}\frac{(K+M)\left(  \ln(K+M)\right)  ^{\frac{3}{2}}}{(1-(\ln
K)^{-1})^{1/2}}%
\end{align*}
or $\left\vert \nabla u(y_{0})\right\vert \leq s_{0}$, leading to
\[
\frac{|\nabla u(x_{0})|}{\ln(K+u(x_{0}))}\leq\frac{g(y_{0})\left\vert \nabla
u(y_{0})\right\vert }{\ln(K+u(y_{0}))}\leq\frac{\left\vert \nabla
u(y_{0})\right\vert }{\ln K},
\]
and so
\[
\left\vert \nabla u(x_{0})\right\vert \leq s_{0}\frac{\ln(K+M)}{\ln K},
\]
or, finally (\ref{eq8}) holds, leading to
\begin{align*}
\frac{|\nabla u(x_{0})|}{\ln(K+u(x_{0}))}  &  \leq\frac{g(y_{0})|\nabla
u(y_{0})|}{\ln(K+u(y_{0}))}\\
&  \leq\frac{4f}{r|f^{\prime}|}\frac{1}{\ln(K+u(y_{0}))}\\
&  =\frac{4}{r}(K+u(y_{0}))
\end{align*}
and so
\[
|\nabla u(x_{0})|\leq\frac{4(K+M)\ln(K+M)}{r}.
\]
We thus proved

\begin{theorem}
\label{locplapl}Let $u$ be a nonnegative solution of (\ref{pde3rep}) in
$B_{r}(x_{0})$, where $r$ is smaller than the distance of $x_{0}$ to its cut
locus, and let Condition \ref{cond10} be satisfied. Then there are constants
$C$ and $K$ depending only on $n$ and the numbers $\alpha$, $\beta$, and
$s_{0}$ in Condition \ref{cond10} such that
\begin{align*}
&  \left\vert \nabla u(x_{0})\right\vert \\
&  \leq C\left\{  1+\frac{1}{r}+\frac{1}{r^{2}}\left(  1+\max_{B_{r}(x_{0}%
)}|\nabla^{2}\rho^{2}|\right)  +\max_{B_{r}(x_{0})}|\operatorname*{Ric}%
\nolimits^{-}|\right\}  ^{1/2}(K+M)\left(  \ln(K+M)\right)  ^{2}%
\end{align*}
with $\displaystyle{M=\max_{B_{r}(x_{0})}u}$.
\end{theorem}

We remark that Conditions \ref{cond6} and \ref{cond10} are for example
fulfilled if $b(s)+1=cs^{m}$ for some $c>0,$ $m\geq0.$ Thus, Theorems
\ref{globplapl} and \ref{locplapl} are valid for the $p-$Laplacian where
$b=p-2,$ $p>1.$

\subsubsection{The case of strong decay of the eigenvalue ratio}

\qquad For the global estimates we may use the same type of auxiliary function
as before, though with a different $f,$ namely, $f(u)=\exp(Ku)$ with a
positive constant $K.$ With $g\equiv1$ we obtain from (\ref{eq5})%
\begin{align*}
0  &  \geq-b^{\prime}|\nabla u|\frac{f^{\prime2}}{f^{2}}|\nabla u|^{2}%
+(b+1)\left(  \frac{f^{\prime\prime}}{f}-2\frac{f^{\prime2}}{f^{2}}\right)
|\nabla u|^{2}+\text{$\operatorname*{Ric}$}(E_{1},E_{1})\\
&  \geq K^{2}\left(  -b^{\prime}\left\vert \nabla u\right\vert -\left(
b+1\right)  \right)  \left\vert \nabla u\right\vert ^{2}-|\operatorname*{Ric}%
\nolimits^{-}|.
\end{align*}

If we introduce the condition:

\begin{condition}
\label{cond11-1}there are positive numbers $\alpha$ and $s_{0}$ such that%
\[
\left(  -b^{\prime}(s)s-\left(  b(s)+1\right)  \right)  s^{2}\geq\alpha,\text{
}s\geq s_{0}%
\]

\end{condition}

\noindent and choose
\[
{K>}\left(  {\frac{1}{\alpha}\max_{\Omega}|\operatorname*{Ric}\nolimits^{-}%
|}\right)  {^{1/2}},
\]
then we immediately get the result:

\begin{theorem}
\label{globmin}If Condition \ref{cond11-1} holds and the function
$\exp(Ku)|\nabla u|$ attains a local maximum in an interior point $y_{0}$ of
$\Omega$ then it follows that $\left\vert \nabla u(y_{0})\right\vert \leq
s_{0}$, where $K$ must be chosen as above.
\end{theorem}

In order to derive an analogous local gradient estimate it seems to be
necessary to slightly modify the auxiliary function $G$ and choose
\[
G(x)=g(x)f(u)\ln|\nabla u|,\;g=1-\frac{\rho^{2}}{r^{2}}.
\]
This choice and the subsequent calculations are inspired by a paper of Wang
\cite{W} which deals with the Euclidean mean curvature equation. It is
sufficient to consider this function in a range where $|\nabla u|>1$ such that
also $\ln G$ is well defined. Again we consider $G$ in the neighborhood of a
point $y_{0}\in\Omega$ where $G$ attains a local maximum. \newline Taking the
function $F(s)=\ln(s)$ in Lemma \ref{lem2}, we obtain at $y_{0}$
\begin{equation}
\frac{\nabla^{2}u(E_{1},E_{i})}{|\nabla u|\ln|\nabla u|}=-\frac{f^{\prime}}%
{f}\langle\nabla u,E_{i}\rangle-\frac{\langle\nabla g,E_{i}\rangle}{g}
\label{af1}%
\end{equation}
and
\begin{align}
0  &  \geq-\frac{1}{|\nabla u|^{2}\ln|\nabla u|}\left[  b^{\prime}|\nabla
u|+(b+1)(1+\frac{1}{\ln|\nabla u|})\right]  \nabla^{2}u(E_{1},E_{1}%
)^{2}\nonumber\\
&  +\sum_{\substack{i,j\\i\geq2}}\frac{\nabla^{2}u(E_{i},E_{j})^{2}}{|\nabla
u|^{2}\ln|\nabla u|}-\frac{1}{|\nabla u|^{2}\ln|\nabla u|}\left[  b+1+\frac
{1}{\ln|\nabla u|}\right]  \sum_{i\geq2}\nabla^{2}u(E_{1},E_{i})^{2}%
\nonumber\\
&  +(b+1)\frac{f^{\prime\prime}}{f}|\nabla u|^{2}+\frac{1}{\ln|\nabla
u|}\text{$\operatorname*{Ric}$}(E_{1},E_{1})\nonumber\\
&  +\frac{1}{g}\left[  (b+1)\nabla^{2}g(E_{1},E_{1})+\sum_{i\geq2}\nabla
^{2}g(E_{i},E_{i})\right] \nonumber\\
&  -(b+1)\left[  \frac{f^{\prime2}}{f^{2}}|\nabla u|^{2}+\frac{1}{g^{2}%
}\langle\nabla g,E_{1}\rangle^{2}\right]  -\frac{1}{g^{2}}\sum_{i\geq2}%
\langle\nabla g,E_{i}\rangle^{2}. \label{af2}%
\end{align}
From (\ref{af1}) we have
\begin{equation}
\frac{\nabla^{2}u(E_{1},E_{1})^{2}}{|\nabla u|^{2}\ln^{2}|\nabla u|}%
=\frac{f^{\prime2}}{f^{2}}|\nabla u|^{2}+\frac{\langle\nabla g,E_{1}%
\rangle^{2}}{g^{2}}+\frac{2f^{\prime}}{fg}\langle\nabla g,E_{1}\rangle|\nabla
u| \label{bf1}%
\end{equation}
and
\begin{equation}
\frac{\nabla^{2}u(E_{1},E_{i})^{2}}{|\nabla u|^{2}\ln^{2}|\nabla u|}%
=\frac{\langle\nabla g,E_{i}\rangle^{2}}{g^{2}},\;\forall i=2,\ldots,n
\label{bf2}%
\end{equation}
Since $b+1\geq0,$ it follows that%
\begin{align}
\sum_{\substack{i,j\\i\geq2}}\nabla^{2}u(E_{i},E_{j})^{2}  &  =\sum_{i\geq
2}\nabla^{2}u(E_{i},E_{1})^{2}+\sum_{\substack{i\geq2\\j\geq2}}\nabla
^{2}u(E_{i},E_{j})^{2}\nonumber\\
&  =-b\sum_{i\geq2}\nabla^{2}u(E_{i},E_{1})^{2}+\left(  b+1\right)
\sum_{i\geq2}\nabla^{2}u(E_{i},E_{1})^{2}\nonumber\\
&  +\sum_{\substack{i\geq2\\j\geq2}}\nabla^{2}u(E_{i},E_{j})^{2}\nonumber\\
&  \geq-b\sum_{i\geq2}\nabla^{2}u(E_{i},E_{1})^{2}. \label{bfsum}%
\end{align}
By (\ref{af2}), (\ref{bf1}), (\ref{bf2}) and (\ref{bfsum}), we have%
\begin{align}
0  &  \geq-\frac{1}{|\nabla u|^{2}\ln|\nabla u|}\left[  b^{\prime}|\nabla
u|+(b+1)(1+\frac{1}{\ln|\nabla u|})\right]  \nabla^{2}u(E_{1},E_{1}%
)^{2}\nonumber\\
&  -\frac{1}{|\nabla u|^{2}\ln|\nabla u|}\left[  2b+1+\frac{1}{\ln|\nabla
u|}\right]  \sum_{i\geq2}\nabla^{2}u(E_{1},E_{i})^{2}+(b+1)\frac
{f^{\prime\prime}}{f}|\nabla u|^{2}\nonumber\\
&  +\frac{1}{\ln|\nabla u|}\text{$\operatorname*{Ric}$}(E_{1},E_{1})+\frac
{1}{g}\left[  (b+1)\nabla^{2}g(E_{1},E_{1})+\sum_{i\geq2}\nabla^{2}%
g(E_{i},E_{i})\right] \nonumber\\
&  -(b+1)\left[  \frac{\nabla^{2}u(E_{1},E_{1})^{2}}{|\nabla u|^{2}\ln
^{2}|\nabla u|}-\frac{2f^{\prime}}{fg}\langle\nabla g,E_{1}\rangle|\nabla
u|\right]  -\sum_{i\geq2}\frac{\nabla^{2}u(E_{1},E_{i})^{2}}{|\nabla u|^{2}%
\ln^{2}|\nabla u|}\nonumber\\
&  =-\frac{1}{|\nabla u|^{2}\ln|\nabla u|}\left[  b^{\prime}|\nabla
u|+(b+1)(1+\frac{2}{\ln|\nabla u|})\right]  \nabla^{2}u(E_{1},E_{1}%
)^{2}\nonumber\\
&  -\frac{1}{|\nabla u|^{2}\ln|\nabla u|}\left[  2b+1+\frac{2}{\ln|\nabla
u|}\right]  \sum_{i\geq2}\nabla^{2}u(E_{1},E_{i})^{2}+(b+1)\frac
{f^{\prime\prime}}{f}|\nabla u|^{2}\nonumber\\
&  +\frac{1}{\ln|\nabla u|}\text{$\operatorname*{Ric}$}(E_{1},E_{1})+\frac
{1}{g}\left[  (b+1)\nabla^{2}g(E_{1},E_{1})+\sum_{i\geq2}\nabla^{2}%
g(E_{i},E_{i})\right] \nonumber\\
&  +\frac{2(b+1)f^{\prime}}{fg}\langle\nabla g,E_{1}\rangle|\nabla u|
\label{ext}%
\end{align}
For a given $\varepsilon>0$ only depending on a structural condition for
(\ref{pde3rep}) we may assume that
\begin{equation}
\ln|\nabla u(y_{0})|\geq\frac{2}{\varepsilon} \label{cond15}%
\end{equation}
because if (\ref{cond15}) does not hold the estimate for $|\nabla u|$ will
turn out to be trivial.

From (\ref{ext}) and (\ref{cond15}) we obtain
\begin{align*}
0  &  \geq-\frac{1}{|\nabla u|^{2}\ln|\nabla u|}\left[  b^{\prime}|\nabla
u|+(b+1)(1+\varepsilon)\right]  \nabla^{2}u(E_{1},E_{1})^{2}\\
&  -\frac{1}{|\nabla u|^{2}\ln|\nabla u|}(2b+1+\varepsilon)\sum_{i\geq2}%
\nabla^{2}u(E_{1},E_{i})^{2}+(b+1)\frac{f^{\prime\prime}}{f}|\nabla u|^{2}\\
&  +\frac{1}{\ln|\nabla u|}\text{$\operatorname*{Ric}$}(E_{1},E_{1})+\frac
{1}{g}\left[  (b+1)\nabla^{2}g(E_{1},E_{1})+\sum_{i\geq2}\nabla^{2}%
g(E_{i},E_{i})\right] \\
&  +\frac{2(b+1)f^{\prime}}{fg}\langle\nabla g,E_{1}\rangle|\nabla u|.
\end{align*}
Let us assume that%
\begin{equation}
2b\left(  \left\vert \nabla u\right\vert \right)  +1+\varepsilon\leq0.
\label{negat}%
\end{equation}
Hence,%
\begin{align}
0  &  \geq-\frac{1}{|\nabla u|^{2}\ln|\nabla u|}\left[  b^{\prime}|\nabla
u|+(b+1)(1+\varepsilon)\right]  \nabla^{2}u(E_{1},E_{1})^{2}\nonumber\\
&  +(b+1)\frac{f^{\prime\prime}}{f}|\nabla u|^{2}\frac{1}{\ln|\nabla
u|}\text{$\operatorname*{Ric}$}(E_{1},E_{1})+\frac{2(b+1)f^{\prime}}%
{fg}\langle\nabla g,E_{1}\rangle|\nabla u|\nonumber\\
&  +\frac{1}{g}\left[  (b+1)\nabla^{2}g(E_{1},E_{1})+\sum_{i\geq2}\nabla
^{2}g(E_{i},E_{i})\right]  . \label{tg15}%
\end{align}
\qquad We shall choose $f$ in such a way that
\[
\frac{f^{\prime}}{f}\geq1
\]
and assume that
\begin{equation}
rg(y_{0})\left\vert \nabla u(y_{0})\right\vert \geq4 \label{tg16}%
\end{equation}
because otherwise
\[
g(y_{0})\ln\left\vert \nabla u(y_{0})\right\vert \leq\frac{1}{4}%
\]
and the estimative of $|\nabla u(x_{0})|$ becomes trivial. By (\ref{tg16}) we
have
\begin{equation}
\left\vert \frac{\langle\nabla g,E_{i}\rangle}{g}\right\vert \leq\frac
{2\rho|\nabla\rho|}{gr^{2}}\leq\frac{2}{gr}\leq\frac{|\nabla u|}{2}.
\label{tr1}%
\end{equation}
It follows then from (\ref{af1}) and (\ref{tr1}) that
\begin{align}
\left\vert \frac{\nabla^{2}u(E_{1},E_{1})}{|\nabla u|\ln|\nabla u|}%
\right\vert  &  \geq\left\vert \frac{f^{\prime}\langle\nabla u,E_{1}\rangle
}{f}\right\vert -\left\vert \frac{\langle\nabla g,E_{1}\rangle}{g}\right\vert
\nonumber\\
&  \geq\frac{f^{\prime}|\nabla u|}{f}-\frac{|\nabla u|}{2}\nonumber\\
&  \geq\frac{f^{\prime}|\nabla u|}{f}-\frac{f^{\prime}|\nabla u|}{2f}%
=\frac{f^{\prime}|\nabla u|}{2f} \label{tr2}%
\end{align}
and
\begin{align}
\frac{2f^{\prime}}{fg}(b+1)\langle\nabla g,E_{1}\rangle|\nabla u|  &
\geq-\frac{2f^{\prime}}{f}(b+1)\left\vert \frac{\left\langle \nabla
g,E_{1}\right\rangle }{g}\right\vert \left\vert \nabla u\right\vert
\nonumber\\
&  \geq-\frac{f^{\prime}}{f}(b+1)|\nabla u|^{2}. \label{tr4}%
\end{align}
From (\ref{tg15}), (\ref{tr2}) and (\ref{tr4}) one obtains
\begin{align*}
0  &  \geq-\frac{f^{\prime2}}{4f^{2}}\left[  b^{\prime}|\nabla
u|+(b+1)(1+\varepsilon)\right]  |\nabla u|^{2}\ln|\nabla u|+(b+1)\left(
\frac{f^{\prime\prime}}{f}-\frac{f^{\prime}}{f}\right)  |\nabla u|^{2}\\
&  +\frac{1}{\ln|\nabla u|}\text{$\operatorname*{Ric}$}(E_{1},E_{1})+\frac
{1}{g}\left[  (b+1)\nabla^{2}g(E_{1},E_{1})+\sum_{i\geq2}\nabla^{2}%
g(E_{i},E_{i})\right]
\end{align*}
where we assumed that
\begin{equation}
b^{\prime}|\nabla u|+(1+\varepsilon)(b+1)\leq0. \label{cond17}%
\end{equation}
Finally we therefore arrive at
\begin{align}
0  &  \geq-\frac{f^{\prime2}}{4f^{2}}\left[  b^{\prime}|\nabla u|+(b+1)\left(
1+\varepsilon\right)  \right]  |\nabla u|^{2}\ln|\nabla u|+(b+1)\left(
\frac{f^{\prime\prime}}{f}-\frac{f^{\prime}}{f}\right)  |\nabla u|^{2}%
\nonumber\\
&  +\frac{1}{\ln|\nabla u|}\text{$\operatorname*{Ric}$}(E_{1},E_{1})+\frac
{1}{g}\left[  (b+1)\nabla^{2}g(E_{1},E_{1})+\sum_{i\geq2}\nabla^{2}%
g(E_{i},E_{i})\right]  . \label{td}%
\end{align}
We now introduce the following hypothesis:

\begin{condition}
\label{cond18-1}There are positive numbers $\varepsilon$, $\alpha$, and
$s_{0}$ such that
\begin{equation}
\left(  -b^{\prime}(s)s-\left(  1+\varepsilon\right)  \left(  b(s)+1\right)
\right)  s^{2}\geq\alpha\text{\ for}\;s\geq s_{0}. \label{bcond18}%
\end{equation}

\end{condition}

\begin{remark}
\label{rob}From (\ref{bcond18}) we have%
\[
-\frac{b^{\prime}(s)}{b(s)+1}>\frac{1+\varepsilon}{s}%
\]
Integrating between $s_{0}$ and $s$ we arrive at%
\[
-\ln\left(  \frac{b(s)+1}{b(s_{0})+1}\right)  >\left(  1+\varepsilon\right)
\ln\left(  \frac{s}{s_{0}}\right)
\]
and so%
\[
\frac{b(s_{0})+1}{b(s)+1}>\left(  \frac{s}{s_{0}}\right)  ^{1+\varepsilon}%
\]
then%
\[
2b(s)+1+\varepsilon<-1+\varepsilon+2\left(  \frac{s_{0}}{s}\right)
^{1+\varepsilon}\left(  b\left(  s_{0}\right)  +1\right)  \rightarrow
-1+\varepsilon<0
\]
when $s\rightarrow+\infty.$ Taking $s_{0}>0$ large enough we arrive at%
\[
2b(s)+1+\varepsilon<0\text{ for }s\geq s_{0}.
\]
It follows from the Condition \ref{cond18-1} that (\ref{negat}) is valid.
\end{remark}

\begin{remark}
In the case $b\left(  s\right)  +1=\frac{1}{s^{2}+1}$ and $\varepsilon<1,$ we
have%
\[
\left(  -b^{\prime}(s)s-\left(  1+\varepsilon\right)  \left(  b(s)+1\right)
\right)  s^{2}=\frac{2s^{4}}{\left(  s^{2}+1\right)  ^{2}}-\frac
{(1+\varepsilon)s^{2}}{s^{2}+1}\rightarrow1-\varepsilon>0
\]
and%
\[
2b\left(  s\right)  +1+\varepsilon=\frac{-2s^{2}}{s^{2}+1}+1+\varepsilon
\rightarrow-1+\varepsilon<0
\]
when $s\rightarrow+\infty.$
\end{remark}

It follows from Condition \ref{cond18-1} that (\ref{cond17}) is also valid and
that $b^{\prime}(s)\leq0$ for $s\geq s_{0}$ and hence there is a number
$\beta\geq1$ such that $b(s)+1\leq\beta$ for $s\geq s_{0}$. We now choose
$f(u)=\exp(Ku)$ with $K>1$. From (\ref{cond15}), (\ref{td}) and the condition
(\ref{cond18-1}) we have
\begin{align*}
0  &  \geq-\frac{K^{2}}{4}\left[  b^{\prime}|\nabla u|+\left(  1+\varepsilon
\right)  (b+1)\right]  |\nabla u|^{2}\ln|\nabla u|\\
&  +(b+1)\left(  K^{2}-K\right)  |\nabla u|^{2}-\frac{1}{\ln|\nabla
u|}|\operatorname*{Ric}\nolimits^{-}|\\
&  -\frac{1}{g}\left[  (b+1)\frac{|\nabla^{2}\rho|}{r^{2}}+\sqrt{n-1}%
\frac{|\nabla^{2}\rho^{2}|}{r^{2}}\right] \\
&  \geq\frac{K^{2}\alpha}{4}\ln|\nabla u|-\frac{\varepsilon}{2}%
|\operatorname*{Ric}\nolimits^{-}|-\frac{\beta+\sqrt{n-1}}{g}\frac{|\nabla
^{2}\rho^{2}|}{r^{2}}%
\end{align*}
We can now conclude that
\[
g(y_{0})\ln|\nabla u(y_{0})|\leq\frac{4}{K^{2}\alpha}\left(  \frac
{\varepsilon}{2}|\operatorname*{Ric}\nolimits^{-}|+\frac{\beta+\sqrt{n-1}%
}{r^{2}}|\nabla^{2}\rho^{2}|\right)  ,
\]
unless
\[
\ln|\nabla u(y_{0})|\leq\max\left\{  \frac{4}{g(y_{0})r},\ln s_{0}\right\}  .
\]

Since by construction $G(x_{0})\leq G(y_{0})$ we therefore either have
\begin{align*}
\exp(Ku(x_{0}))\ln|\nabla u(x_{0})|  &  \leq g(y_{0})\exp(Ku(y_{0}))\ln|\nabla
u(y_{0})|\\
&  \leq\frac{4\exp(Ku(y_{0}))}{K^{2}\alpha}\left(  \frac{\varepsilon}%
{2}|\operatorname*{Ric}\nolimits^{-}|+\frac{\beta+\sqrt{n-1}}{r^{2}}%
|\nabla^{2}\rho^{2}|\right)  .
\end{align*}
Hence there is a constant $C=C(n,\beta,\alpha,\varepsilon)$ such that
\begin{align*}
\ln|\nabla u(x_{0})|  &  \leq C\frac{\exp(KM)}{K^{2}}\max\left\{
|\operatorname*{Ric}\nolimits^{-}|+\frac{1}{r^{2}}|\nabla^{2}\rho^{2}|\right\}
\\
&  \leq C\exp(KM)\max\left\{  |\operatorname*{Ric}\nolimits^{-}|+\frac
{1}{r^{2}}|\nabla^{2}\rho^{2}|\right\}  ,
\end{align*}
where
\[
M=\max_{x\in B_{r}(x_{0})}|u(x)-u(x_{0})|
\]
\newline or
\[
g(y_{0})\ln|\nabla u(y_{0})|\leq\frac{4}{r},
\]
leading to
\[
\exp(Ku(x_{0}))\ln|\nabla u(x_{0})|\leq g(y_{0})\exp(Ku(y_{0}))\ln|\nabla
u(x_{0})|\leq\frac{4}{r}\exp(Ku(y_{0}))
\]
and thus
\[
\ln|\nabla u(x_{0})|\leq\frac{4}{r}\exp(KM)
\]
or, finally $|\nabla u(y_{0})|\leq s_{0}$ holds, leading to
\[
\exp(Ku(x_{0}))\ln|\nabla u(x_{0})|\leq g(y_{0})\exp(Ku(y_{0}))\ln s_{0}%
\leq\exp(Ku(y_{0}))\ln s_{0}%
\]
and so
\[
\ln|\nabla u(x_{0})|\leq\exp(KM)\ln s_{0}.
\]
We thus proved

\begin{theorem}
\label{locmin}Let $u$ be a solution of (\ref{pde3rep}) in $\Omega$ such that
the Condition \ref{cond18-1} is satisfied and let $x_{0}\in\Omega$ and $r>0$
such that the geodesic ball $B_{r}(x_{0})$, is contained in $\Omega$ and is
disjoint of the cut locus of $x_{0}$. Then there is a constant $C$ depending
only $n$ and the numbers $\varepsilon$, $\alpha$, and $s_{0}$ in the Condition
\ref{cond18-1} such that
\[
\ln\left\vert \nabla u(x_{0})\right\vert \leq C\exp(KM)\max\left\{  1+\frac
{1}{r}+\max\left(  |\operatorname*{Ric}\nolimits^{-}|+\frac{1}{r^{2}}%
|\nabla^{2}\rho^{2}|\right)  \right\}
\]
for all $K>1,$ where
\[
{M=\max_{x\in B_{r}(x_{0})}|u(x)-u(x_{0})|}.
\]

\end{theorem}

We remark that Condition \ref{cond11-1} and Condition \ref{cond18-1}, the
latter with $\varepsilon=1/4,$ are satisfied if
\[
b(s)+1=\frac{1}{1+s^{2}}.
\]
Hence, Theorems \ref{locmin} and \ref{globmin} are valid for the minimal
surface equation.

\newpage

\section{\label{et}Existence theorems for regular equations on bounded
domains}

\qquad In the course of the proof of the existence theorems we shall make
references to results of \cite{GT} by using local coordinate systems.

Let $\mathcal{M}$ be a complete Riemannian manifold. We note that in local
coordinates $\left\{  \partial/\partial x_{i},\text{ }1\leq i\leq n\right\}  $
of $T\Omega\subset T\mathcal{M}$ defined on an open set $U$ of $\mathbb{R}%
^{n}$, the functional%
\[
F\left(  u\right)  =\int_{\Omega}\phi\left(  \left\vert \nabla u\right\vert
\right)  \omega
\]
can be written as%
\[
F\left(  u\right)  =\int_{U}\phi\left(  \left\vert \nabla u\right\vert
\right)  \sqrt{g}dx
\]
where $g=\det\left(  g_{ij}\right)  ,$ $g_{ij}=\left\langle \partial/\partial
x_{i},\partial/\partial x_{j}\right\rangle ,$ $\left(  g^{ij}\right)  =\left(
g_{ij}\right)  ^{-1}.$ We may see that given $l\in\left\{  1,...,n\right\}  $
the $l-$local component of $\left(  \nabla u\right)  ^{l}$ is
\[
\left(  \nabla u\right)  ^{l}=g^{kl}\frac{\partial u}{\partial x_{k}}%
\]
where the summation is understood and so%
\begin{equation}%
\begin{array}
[c]{c}%
\left\vert \nabla u\right\vert ^{2}=g_{ij}\left(  \nabla u\right)  ^{i}\left(
\nabla u\right)  ^{j}=g_{ij}g^{ik}\frac{\partial u}{\partial x_{k}}g^{jl}%
\frac{\partial u}{\partial x_{l}}\\
\text{ }\\
=\delta_{jk}\frac{\partial u}{\partial x_{k}}g^{jl}\frac{\partial u}{\partial
x_{l}}=g^{jl}\frac{\partial u}{\partial x_{k}}\frac{\partial u}{\partial
x_{l}}.
\end{array}
\label{gradcor}%
\end{equation}

To obtain the Euler-Lagrange equations we note that, writing
\[
A\left(  \left\vert \nabla u\right\vert \right)  =\frac{a\left(  \left\vert
\nabla u\right\vert \right)  }{\left\vert \nabla u\right\vert },
\]
we have for $\varphi\in C_{0}^{\infty}\left(  U\right)  $%
\begin{align*}
\int_{U}\frac{\phi^{\prime}\left(  \left\vert \nabla u\right\vert \right)
}{\left\vert \nabla u\right\vert }\left\langle \nabla u,\nabla\varphi
\right\rangle \sqrt{g}dx  &  =\int_{U}A\left(  \left\vert \nabla u\right\vert
\right)  g^{ij}\frac{\partial u}{\partial x_{i}}\frac{\partial\varphi
}{\partial x_{j}}\sqrt{g}dx\\
&  =-\int_{U}\frac{\partial}{\partial x_{j}}\left(  \sqrt{g}g^{ij}A\left(
\left\vert \nabla u\right\vert \right)  \frac{\partial u}{\partial x_{i}%
}\right)  \varphi dx
\end{align*}
so that%
\begin{equation}%
{\displaystyle\sum\limits_{j}}
\frac{\partial}{\partial x_{j}}A^{j}\left(  x,Du\right)  =0 \label{el}%
\end{equation}
where
\[
A^{j}=%
{\displaystyle\sum\limits_{i}}
\left(  \sqrt{g}g^{ij}A\left(  \left\vert \nabla u\right\vert \right)
\frac{\partial u}{\partial x_{i}}\right)  .
\]

\subsection{The case of mild decay of the eigenvalue ratio}

We begin with the case of smooth boundary data.

\begin{theorem}
\label{exreg}Let $\Omega\subset\mathcal{M}$ be a domain of class $C^{2,\alpha
}$ for some $\alpha>0,$ $\overline{\Omega}$ compact and let $g\in C^{2,\alpha
}\left(  \overline{\Omega}\right)  .$ We write $a$ as $a(s)=sA(s)$ and assume

\begin{itemize}
\item[(i)] $A\in C^{1,\alpha}\left(  \left[  0,\infty\right)  \right)  \cap
C^{2,\alpha}\left(  \left(  0,\infty\right)  \right)  ,$%
\begin{equation}
\min_{0\leq s\leq s_{0}}\left\{  A,\text{ }1+\frac{sA^{\prime}(s)}%
{A(s)}\right\}  >0 \label{min1}%
\end{equation}
for any $s_{0}>0$

\item[(ii)] there is a non-decreasing function $\varphi$ defined on some
interval $\left[  s_{0},\infty\right)  $ such that
\[
\int_{s_{0}}^{\infty}\frac{\varphi(s)}{s^{2}}ds=\infty
\]
and%
\[
\left(  1+b^{-}(s)\right)  s^{2}\geq\varphi(s),
\]
where
\[
b(s)=\frac{sa^{\prime}(s)}{a(s)}-1=\frac{sA^{\prime}(s)}{A(s)}%
\]

\item[(iii)] there are a $C^{1}$ function $\psi$ on some interval $\left[
s_{0},\infty\right)  $ with $\psi\left(  s\right)  \rightarrow\infty$ as
$s\rightarrow\infty$ and $\beta>0$ such that%
\[
\left(  1+b(s)-\beta\left(  b^{\prime}\right)  ^{+}(s)s\right)  s^{2}\geq
\psi(s).
\]
Then the Dirichlet problem
\begin{equation}
\left\{
\begin{array}
[c]{l}%
Q\left[  u\right]  =\operatorname{div}\left(  A\left(  \left\vert \nabla
u\right\vert \right)  \nabla u\right)  =0\text{ in }\Omega,\\
u|\partial\Omega=g
\end{array}
\right.  \label{dir}%
\end{equation}
has a unique solution $u\in C^{2,\alpha}\left(  \overline{\Omega}\right)  .$
\end{itemize}
\end{theorem}

\begin{proof}
As explained in Subsection \ref{mapbs}, it is enough to obtain a priori
gradient bounds for a one parameter family of solutions of (\ref{dir}).
Precisely, it is enough to prove that there is a constant $C$ depending only
on $\Omega,$ $A$ and $g$ such that%
\begin{equation}
\left\vert \nabla v\right\vert \leq N \label{ub}%
\end{equation}
if $v\in C^{2}\left(  \overline{\Omega}\right)  $ satisfies $Q\left[
v\right]  =0$ and $v|\partial\Omega=\sigma g$ for some $\sigma\in\left[
0,1\right]  .$ But this is immediate from (ii) and (iii) and Theorems
\ref{bgem} and \ref{globplapl}. This proves Theorem \ref{exreg}.
\end{proof}

\begin{remark}
The assumption of $A\in C^{2,\alpha}\left(  \left(  0,\infty\right)  \right)
$ is necessary to guarantee that $u$ is of class $C^{3}$ (in the set
$\left\vert \nabla u\right\vert >s_{0}>0$ for large $s_{0})$ and to apply the
results of Section \ref{glge}.
\end{remark}

\begin{theorem}
\label{excbd}Let $\Omega$ be as in Theorem \ref{exreg} and $g\in C^{0}\left(
\partial\Omega\right)  .$ We assume Condition (i) as in Theorem \ref{exreg}
but instead of (ii) and (iii) as in this theorem we require the stronger condition%

\begin{equation}
B(s)^{-1}\left(  1+b(s)-\beta\left\vert b^{\prime}(s)\right\vert s\right)
\geq\alpha\label{besao}%
\end{equation}
with some positive constant $\alpha$ and%
\[
B(s):=\max\left\{  1,1+b(s)\right\}  ,\text{ }s>0.
\]

Then the above Dirichlet problem has a unique solution $u\in C^{2,\alpha
}\left(  \Omega\right)  \cap C^{0}\left(  \overline{\Omega}\right)  .$
Moreover, for any relatively compact subdomain $\Lambda$ of $\Omega$ there is
a $C^{2,\alpha}\left(  \Lambda\right)  $ bound depending only on $\Lambda$ and
$\sup_{\Omega}\left\vert u\right\vert .$
\end{theorem}

\begin{proof}
Let us first remark that condition (\ref{besao}) implies conditions (ii) and
(iii) in Theorem \ref{exreg}. Indeed, since $B(s)\geq1,$ the condition (iii)
is clearly satisfied with $\varphi(s)=\alpha s^{2}.$ Recalling that
\[
\frac{1+b}{B}=1+b^{-},
\]
it follows from (\ref{besao}) that $1+b^{-}\geq\alpha$ and hence (ii) is
fulfilled with $\varphi(s)=\alpha s^{2}.$ We therefore have Theorem
\ref{exreg} at our disposal and then, with Theorem \ref{locplapl}, it is
enough to apply the technique explained at Section \ref{mapbc} to conclude the
proof of the theorem.
\end{proof}

\subsection{The case of strong decay of the eigenvalue ratio}

\begin{theorem}
\label{exminreg}Let $\Omega$ be a bounded domain of class $C^{2,\alpha}$ in
$\mathcal{M}$ such that the mean curvature of $\partial\Omega$ with respect to
the interior normal of $\partial\Omega$ as well as of the inner parallel
hypersurfaces of $\partial\Omega$ in some neighborhood of $\partial\Omega$ is
non negative. We write $a$ as $a(s)=sA(s)$ and assume

\begin{itemize}
\item[(i)] $A\in C^{1,\alpha}\left(  \left[  0,\infty\right)  \right)  \cap
C^{2,\alpha}\left(  \left(  0,\infty\right)  \right)  ,$%
\begin{equation}
\min_{0\leq s\leq s_{0}}\left\{  A,\text{ }1+\frac{sA^{\prime}(s)}%
{A(s)}\right\}  >0 \label{mmin1}%
\end{equation}
for any $s_{0}>0$

\item[(ii)] there are positive numbers $\alpha$ and $s_{0}$ such that%
\begin{equation}
\left(  -b^{\prime}(s)s-\left(  b(s)+1\right)  \right)  s^{2}\geq\alpha,\text{
}s\geq s_{0}. \label{ccc}%
\end{equation}

\end{itemize}

Then the Dirichlet problem
\[
\left\{
\begin{array}
[c]{l}%
Q\left[  u\right]  =\operatorname{div}\left(  A\left(  \left\vert \nabla
u\right\vert \right)  \nabla u\right)  =0\text{ in }\Omega,\\
u|\partial\Omega=g
\end{array}
\right.
\]
where $g\in C^{2,\alpha}\left(  \overline{\Omega}\right)  $ has a unique
solution $u\in C^{2,\alpha}\left(  \overline{\Omega}\right)  .$
\end{theorem}

\begin{proof}
Using our Theorems \ref{bgem} and \ref{globmin} the proof of the above theorem
is completely analogous to that of Theorem \ref{exreg} once we prove that
condition (ii) implies Condition II in the introduction.

Clearly (\ref{ccc}) implies that $b^{\prime}\leq0$ for $s\geq s_{0}$ and hence
$b$ is non increasing for $s\geq s_{0}.$ If there were numbers $C>0$ and
$s_{1}>0$ such that $-b^{\prime}(s)s\geq C$ for $s\geq s_{1}$ then
$b(s)\rightarrow-\infty$ $\left(  s\rightarrow\infty\right)  ,$ a
contradiction since $b\geq-1.$ Therefore, there is a sequence $s_{k}%
\rightarrow+\infty$ such that $-b(s_{k})s_{k}\rightarrow0$ for $k\rightarrow
\infty.$ Then condition (\ref{ccc}) implies
\[
1+b(s_{k})\leq-b^{\prime}(s_{k})s_{k}\rightarrow0
\]
and so
\begin{equation}
1+b(s)\rightarrow0 \label{b0}%
\end{equation}
as $s\rightarrow\infty$ since $b$ is non increasing. Now condition (\ref{ccc})
also implies
\[
b^{\prime}(s)\leq-\frac{\alpha}{s^{3}}%
\]
for $s\geq s_{0}.$ Therefore we have for $s_{0}<s<\sigma$%
\[
b(\sigma)-b(s)\leq-\alpha\int_{s}^{\sigma}\frac{dt}{t^{3}}=\frac{\alpha}%
{2}\left(  \frac{1}{\sigma^{2}}-\frac{1}{s^{2}}\right)  .
\]
Since $b(\sigma)\rightarrow-1$ ($\sigma\rightarrow\infty)$ by (\ref{b0}) we
obtain
\[
\left(  1+b(s)\right)  s^{2}\geq\frac{\alpha}{2}>0\text{ }\left(
s>s_{0}\right)  .
\]

Since $b(s)<0$ for all sufficient large $s$ we have
\[
\left(  1+b^{-}(s)\right)  s^{2}\geq\frac{\alpha}{2}>0
\]
and so Condition II is satisfied with $\varphi(s)=\alpha/2.$
\end{proof}

\bigskip

Since the condition (\ref{intgradest}) below obviously implies condition
(\ref{ccc}), the proof of the next theorem is completely analogous to the
proof of Theorem \ref{excbd} using the local gradient estimates of Theorem
\ref{locmin}:

\begin{theorem}
\label{exmincbd}Let $\Omega$ be as in the previous theorem$.$ We assume
Condition (i) as in Theorem \ref{exminreg} but instead of (ii) as in this
theorem we require the stronger condition%
\begin{equation}
\left(  -b^{\prime}(s)s-\left(  1+\varepsilon\right)  \left(  b(s)+1\right)
\right)  s^{2}\geq\alpha\ for\;s\geq s_{0} \label{intgradest}%
\end{equation}
for some positive numbers $\varepsilon,$ $\alpha$ and $s_{0}.$ Then the
Dirichlet problem
\[
\left\{
\begin{array}
[c]{l}%
Q\left[  u\right]  =\operatorname{div}\left(  A\left(  \left\vert \nabla
u\right\vert \right)  \nabla u\right)  =0\text{ in }\Omega,\\
u|\partial\Omega=g
\end{array}
\right.
\]
where $g\in C^{0}\left(  \overline{\Omega}\right)  $ has a unique solution
$u\in C^{0}\left(  \overline{\Omega}\right)  \cap C^{2,\alpha}\left(
\Omega\right)  .$
\end{theorem}

\section{Existence theorems for some degenerate or singular equations on
bounded domains}

\qquad In order to prove similar existence theorem for certain singular or
degenerate equations like the $p-$Laplace equations we show how such equations
may be approximated by regular ones satisfying the conditions of Theorems
\ref{exreg} and \ref{excbd}. Precisely, we now assume that the coefficient $a$
is of the form%

\begin{equation}%
\begin{array}
[c]{l}%
a(s)=s^{p-1}A(s),\text{ }s\geq0,\\
\text{where }p>1,\text{ }A\in C^{1,\alpha}\left(  \left[  0,\infty\right)
\right)  \cap C^{2,\alpha}\left(  0,\infty\right)  ,\text{ }A(s)>0\text{ for
}s\geq0.
\end{array}
\label{pltype}%
\end{equation}

We observe that the equation (\ref{PDE1}) becomes singular for $p<2$ and
degenerate elliptic for $p>2.$ We now regularize (\ref{pltype}) in the form
\begin{equation}
a_{\kappa}(s)=\left(  \kappa+s^{2}\right)  ^{\frac{p}{2}-1}A(s)s \label{rega}%
\end{equation}
where $\kappa>0.$ We compute%
\[
a_{\kappa}^{\prime}=\left(  \kappa+s^{2}\right)  ^{\frac{p}{2}-1}A+\left(
\kappa+s^{2}\right)  ^{\frac{p}{2}-1}A^{\prime}s+\left(  \frac{p}{2}-1\right)
\left(  \kappa+s^{2}\right)  ^{\frac{p}{2}-2}2s^{2}A,
\]%
\[
1+b_{\kappa}=\frac{sa_{\kappa}^{\prime}}{a_{\kappa}}=\left(  \kappa
+s^{2}\right)  ^{1-\frac{p}{2}}A^{-1}a_{\kappa}^{\prime}=1+\frac{sA^{\prime}%
}{A}+\left(  p-2\right)  \frac{s^{2}}{\kappa+s^{2}}.
\]

Since%
\[
a^{\prime}=\left(  p-1\right)  s^{p-2}A+s^{p-1}A^{\prime}%
\]
we get%
\begin{equation}
\frac{sa^{\prime}}{a}=\frac{sa^{\prime}}{s^{p-1}A}=\left(  p-1\right)
+\frac{sA^{\prime}}{A} \label{ii}%
\end{equation}
and hence%
\begin{align*}
1+b_{\kappa}  &  =1+\frac{sa^{\prime}}{a}-\left(  p-1\right)  +\left(
p-2\right)  \frac{s^{2}}{\kappa+s^{2}}\\
&  =1+b-\left(  p-2\right)  \left(  1-\frac{s^{2}}{\kappa+s^{2}}\right)
\end{align*}
that is%
\begin{equation}
1+b_{\kappa}=1+b-\left(  p-2\right)  \frac{\kappa}{\kappa+s^{2}}. \label{iii}%
\end{equation}
It follows that%
\begin{equation}
b_{\kappa}^{\prime}=b^{\prime}+\left(  p-2\right)  \frac{2\kappa s}{\left(
\kappa+s^{2}\right)  ^{2}}. \label{iv}%
\end{equation}

Now we would like to check that the Conditions I, \ref{cond6}, \ref{cond10}
hold for the family $a_{\kappa}$ uniformly for $\kappa\in\left(  0,1\right]  $
provided that they are true for the original $a.$

We have from Condition I and (\ref{iii})%
\begin{align*}
\left(  1+b_{\kappa}^{-}\right)  s^{2}  &  \geq\left(  1+b\right)
s^{2}+\left(  p-2\right)  ^{-}\frac{\kappa s^{2}}{\kappa+s^{2}}\\
&  \geq\varphi(s)+\left(  p-2\right)  ^{-}=:\widetilde{\varphi}(s),\text{
}0\leq\kappa\leq1,
\end{align*}
where $\widetilde{\varphi}$ is non decreasing as $\varphi$ is and%
\[
\int_{1}^{\infty}\frac{\widetilde{\varphi}(s)}{s^{2}}ds=\infty
\]
by Condition I. As next we turn to the analogue of Condition \ref{cond6}. We
have%
\[
sb_{\kappa}^{\prime}=sb^{\prime}+\left(  p-2\right)  \frac{2\kappa s^{2}%
}{\left(  \kappa+s^{2}\right)  ^{2}}%
\]
so that%
\[
s\left(  b_{\kappa}^{\prime}\right)  ^{+}\leq s\left(  b^{\prime}\right)
^{+}+\left(  p-2\right)  ^{+}\frac{2\kappa s^{2}}{\left(  \kappa+s^{2}\right)
^{2}}.
\]

It follows then from Condition \ref{cond6} that%
\begin{gather*}
\left(  1+b_{\kappa}(s)-\beta\left(  b_{\kappa}^{\prime}(s)\right)
^{+}s\right)  s^{2}\\
\geq\left(  1+b(s)-\beta\left(  b_{\kappa}^{\prime}(s)\right)  ^{+}s\right)
s^{2}-\left(  p-2\right)  ^{+}\left(  \frac{\kappa s^{2}}{\kappa+s^{2}}%
+\beta\frac{2\kappa s^{4}}{\left(  \kappa+s^{2}\right)  ^{2}}\right) \\
\geq\varphi(s)-\left(  p-2\right)  ^{+}\left(  \kappa+2\beta\kappa\right)
\rightarrow\infty\text{ }\left(  s\rightarrow\infty\right)
\end{gather*}
uniformly for $\kappa\in\left(  0,1\right]  .$ Finally we come to Condition
\ref{cond10}. There holds%
\begin{align*}
1  &  \leq B_{\kappa}=\max\left\{  1,1+b_{k}\right\}  \leq\max\left\{
1,1+b\right\}  +\left(  2-p\right)  ^{+}\frac{\kappa}{\kappa+s^{2}}\\
&  \leq B\left(  1+\left(  2-p\right)  ^{+}\frac{\kappa}{\kappa+s^{2}}\right)
\leq2B
\end{align*}
since $p>1.$ It follows from Condition \ref{cond10} and (\ref{iv}) above that
\begin{gather*}
B_{\kappa}^{-1}\left(  1+b_{\kappa}(s)-\beta\left\vert b_{\kappa}^{\prime
}(s)\right\vert s\right) \\
\geq\frac{1}{2}B^{-1}\left(  1+b(s)-\beta\left\vert b^{\prime}(s)\right\vert
s\right) \\
-\left(  p-2\right)  ^{+}\frac{\kappa}{\kappa+s^{2}}+\beta\left\vert
p-2\right\vert \frac{2\kappa s^{2}}{\left(  \kappa+s^{2}\right)  ^{2}}%
\geq\frac{1}{4}\alpha
\end{gather*}
if $s\geq s_{0}(p,\beta)$ and $\kappa\in\left(  0,1\right]  .$ We thus showed
that the Conditions I, \ref{cond6} and \ref{cond10} imply the corresponding
conditions for the regularized equation, uniformly for $\kappa\in\left(
0,1\right]  .$

Finally we come to Condition \ref{cond18-1} which clearly includes Condition
\ref{cond11-1}, the latter also implying Condition II in the introduction as
was shown in the proof of Theorem \ref{exminreg}. We have%
\begin{gather*}
\left(  -sb_{\kappa}^{\prime}-\left(  1+\varepsilon\right)  \left(
1+b_{\kappa}\right)  \right)  s^{2}=\left(  -sb^{\prime}-\left(
1+\varepsilon\right)  \left(  1+b\right)  \right)  s^{2}\\
+\left(  2-p\right)  \left[  \frac{2\kappa s^{4}}{\left(  \kappa+s^{2}\right)
^{2}}-\left(  1+\varepsilon\right)  \frac{\kappa s^{2}}{\kappa+s^{2}}\right]
\\
\geq\left(  -sb^{\prime}-\left(  1+\varepsilon\right)  \left(  1+b\right)
\right)  s^{2}-\left\vert 2-p\right\vert \kappa\left(  2+1+\varepsilon\right)
\geq\frac{\alpha}{2}%
\end{gather*}
for $s\geq s_{0},$ uniformly for $\kappa\in\left(  0,\kappa_{0}\left(
p\right)  \right)  ,$ $0<\varepsilon\leq1,$ provided that Condition
\ref{cond18-1} holds.

Let us finally check for the ellipticity of the regularized equation. We
compute%
\begin{gather*}
\operatorname{div}\left(  \frac{a_{\kappa}\left(  \left\vert \nabla
u\right\vert \right)  }{\left\vert \nabla u\right\vert }\nabla u\right)  =\\
\left(  \kappa+\left\vert \nabla u\right\vert ^{2}\right)  ^{\frac{p}{2}%
-1}A\Delta u+\left(  \frac{p}{2}-1\right)  \left(  \kappa+\left\vert \nabla
u\right\vert ^{2}\right)  ^{\frac{p}{2}-2}A\left\langle \nabla\left\vert
\nabla u\right\vert ^{2},\nabla u\right\rangle \\
+\left(  \kappa+\left\vert \nabla u\right\vert ^{2}\right)  ^{\frac{p}{2}%
-1}A^{\prime}\left\langle \nabla\left\vert \nabla u\right\vert ,\nabla
u\right\rangle =\\
\left(  \kappa+\left\vert \nabla u\right\vert ^{2}\right)  ^{\frac{p}{2}%
-1}A\Delta u+\left[  \left(  p-2\right)  \left(  \kappa+\left\vert \nabla
u\right\vert ^{2}\right)  ^{\frac{p}{2}-2}A\right. \\
\left.  +\left(  \kappa+\left\vert \nabla u\right\vert ^{2}\right)  ^{\frac
{p}{2}-1}\frac{A^{\prime}}{\left\vert \nabla u\right\vert }\right]  \nabla
^{2}u\left(  \nabla u,\nabla u\right)  ,
\end{gather*}
that is%
\begin{gather*}
\operatorname{div}\left(  \frac{a_{\kappa}\left(  \left\vert \nabla
u\right\vert \right)  }{\left\vert \nabla u\right\vert }\nabla u\right)  =\\
\left(  \kappa+\left\vert \nabla u\right\vert ^{2}\right)  ^{\frac{p}{2}%
-1}A\left\{  \Delta u+\left[  \left(  p-2\right)  \frac{\left\vert \nabla
u\right\vert ^{2}}{\kappa+\left\vert \nabla u\right\vert ^{2}}+\frac
{A^{\prime}\left\vert \nabla u\right\vert }{A}\right]  \nabla^{2}u\left(
\frac{\nabla u}{\left\vert \nabla u\right\vert },\frac{\nabla u}{\left\vert
\nabla u\right\vert }\right)  \right\}  .
\end{gather*}

In order to test the ellipticity we replace $\nabla^{2}u$ by $\left\langle
\xi,\text{ . }\right\rangle \otimes\left\langle \xi,\text{ . }\right\rangle $
with a unit length vector $\xi$ and with%
\[
\eta:=\left\vert \nabla u\right\vert ^{-1}\nabla u,\text{ }\xi^{\bot}%
:=\xi-\left\langle \xi,\eta\right\rangle \eta
\]
we must consider the quadratic form%
\[
q(s,\xi)=\left(  \kappa+s^{2}\right)  ^{\frac{p}{2}-1}A(s)\left\{  \left\vert
\xi^{\bot}\right\vert ^{2}+\left[  1+\left(  p-2\right)  \frac{s^{2}}%
{\kappa+s^{2}}+\frac{sA^{\prime}(s)}{A(s)}\right]  \left\langle \xi
,\eta\right\rangle ^{2}\right\}  .
\]
Since
\[
1+\left(  p-2\right)  \frac{s^{2}}{\kappa+s^{2}}\geq\min\left\{
1,p-1\right\}
\]
we obtain

\begin{lemma}
\label{nnn}We assume that for some interval $\left[  0,s_{0}\right]  $ there
are positive numbers $c,$ $C$ such that%
\begin{align*}
c  &  \leq A(s)\leq C,\\
c  &  \leq\min\left\{  1,p-1\right\}  +\frac{sA^{\prime}(s)}{A(s)}%
\leq1+\left(  p-2\right)  ^{+}+\frac{sA^{\prime}(s)}{A(s)}\leq C,
\end{align*}
for $s\in\left[  0,s_{0}\right]  .$ Then the estimate%
\[
\left(  \kappa+s^{2}\right)  ^{\frac{p}{2}-1}c^{2}\leq q\left(  s,\xi\right)
\leq\left(  \kappa+s^{2}\right)  ^{\frac{p}{2}-1}C^{2}%
\]
holds for $s\in\left[  0,s_{0}\right]  $ and $\left\vert \xi\right\vert =1.$
\end{lemma}

Lemma \ref{nnn} gives the ellipticity condition for the regularized equations
and we may therefore get the following immediate consequences of Theorems
\ref{exreg} and \ref{excbd}. Since we showed above that the Condition I,
\ref{cond6} and \ref{cond10} hold uniformly for $\kappa\in\left(  0,1\right]
$ and Condition \ref{cond18-1} uniformly for $\kappa\in\left(  0,\kappa
_{0}\left(  p\right)  \right]  $ we get $C^{1}$ bounds for the solutions of
the regularized equations independent of $\kappa\in\left(  0,1\right]  $ and
$\kappa\in\left(  0,\kappa_{0}\left(  p\right)  \right]  ,$ respectively.

\begin{corollary}
[of Theorem \ref{exreg}]\label{cor24}We assume all the conditions of Theorem
\ref{exreg} but replacing (\ref{min1}) by
\begin{equation}
\min_{0\leq s\leq s_{0}}\left\{  A(s),\min\left\{  1,p-1\right\}
+\frac{sA^{\prime}(s)}{A(s)}\right\}  >0 \label{nmin1}%
\end{equation}
for any by $s_{0}>0$. Then the regularized Dirichlet problem%
\[
\left\{
\begin{array}
[c]{l}%
\operatorname{div}\left(  \frac{a_{\kappa}\left(  \left\vert \nabla
u\right\vert \right)  }{\left\vert \nabla u\right\vert }\nabla u\right)  =0,\\
u|\partial\Omega=g
\end{array}
\right.
\]
where $a_{\kappa}$ is given by (\ref{rega}), has a unique solution $u_{\kappa
}\in C^{2,\alpha}\left(  \overline{\Omega}\right)  $ for all $\kappa>0.$
Moreover, the family $u_{\kappa},$ $0<\kappa\leq1,$ is uniformly bounded in
$C^{1}\left(  \overline{\Omega}\right)  .$
\end{corollary}

\begin{corollary}
[of Theorem \ref{excbd}]\label{cor23}We assume all the conditions of Theorem
\ref{excbd} but replacing (\ref{min1}) by (\ref{nmin1}). Then the regularized
Dirichlet problem%
\[
\left\{
\begin{array}
[c]{l}%
\operatorname{div}\left(  \frac{a_{\kappa}\left(  \left\vert \nabla
u\right\vert \right)  }{\left\vert \nabla u\right\vert }\nabla u\right)  =0,\\
u|\partial\Omega=g
\end{array}
\right.
\]
where $a_{\kappa}$ is given by (\ref{rega}), has a unique solution $u_{\kappa
}\in C^{2,\alpha}\left(  \Omega\right)  \cap C^{0}\left(  \overline{\Omega
}\right)  ,$ for all $\kappa>0.$ Moreover, the family $u_{\kappa},$
$0<\kappa\leq1,$ has a uniform $C^{1}$ bound on each compact subset of
$\Omega.$
\end{corollary}

\begin{corollary}
[of Theorem \ref{exminreg}]\label{cormmi1}We assume all the conditions of
Theorem \ref{exminreg} but with $a(s)=s^{p-1}A(s)$ and replacing (\ref{mmin1})
by (\ref{nmin1}). Then the regularized Dirichlet problem%
\[
\left\{
\begin{array}
[c]{l}%
\operatorname{div}\frac{a_{\kappa}\left(  \left\vert \nabla u\right\vert
\right)  }{\left\vert \nabla u\right\vert }\nabla u=0\text{ in }\Omega\\
u|\partial\Omega=g
\end{array}
\right.
\]
has a unique solution $u\in C^{2,\alpha}\left(  \overline{\Omega}\right)  $
for all $\kappa\in\left(  0,\kappa_{0}(p)\right)  $ and for all $g\in
C^{2,\alpha}\left(  \overline{\Omega}\right)  .$ Moreover, there is a
$C^{1}\left(  \overline{\Omega}\right)  $ bound for $u$ independent of
$\kappa.$
\end{corollary}

\begin{corollary}
[of Theorem \ref{exmincbd}]\label{cormin2}We assume all the conditions of
Theorem \ref{exmincbd} but with $a(s)=s^{p-1}A(s)$ and replacing (\ref{mmin1})
by (\ref{nmin1}). Then the regularized Dirichlet problem%
\[
\left\{
\begin{array}
[c]{l}%
\operatorname{div}\frac{a_{\kappa}\left(  \left\vert \nabla u\right\vert
\right)  }{\left\vert \nabla u\right\vert }\nabla u=0\text{ in }\Omega\\
u|\partial\Omega=g
\end{array}
\right.
\]
has a unique solution $u\in C^{2,\alpha}\left(  \Omega\right)  \cap
C^{0}\left(  \overline{\Omega}\right)  $ for all $\kappa\in\left(
0,\kappa_{0}(p)\right)  $ and for all $g\in C^{0}\left(  \overline{\Omega
}\right)  .$ Moreover, the family $u_{\kappa},$ $\kappa\in\left(  0,\kappa
_{0}(p)\right)  ,$ admits a uniform $C^{1}$ bound on each relatively compact
open subset $\Lambda$ of $\Omega.$
\end{corollary}

We are now able to obtain existence theorems when the coefficient $a$ behaves
as $a(s)=s^{p-1}A(s)$ (for details see (\ref{pltype})). As an intermediate
step we first demonstrate how from the results obtained so far one may easily
derive the existence of Lipschitz continuous weak solutions in the degenerate
or singular case. A much deeper analysis is required if one wants to show the
optimal $C^{1,\beta}$ regularity of solutions. A complete proof of the latter
would go far beyond the intentions of these notes and we shall therefore
confine ourselves with a reduction to known results in the literature.

\begin{theorem}
We assume either the conditions of Theorem \ref{exreg} or Theorem \ref{excbd},
but replacing (\ref{min1}) and (\ref{mmin1}) by (\ref{nmin1}). Then the
Dirichlet problem%
\[
\left\{
\begin{array}
[c]{l}%
\operatorname{div}\left(  \frac{a\left(  \left\vert \nabla u\right\vert
\right)  }{\left\vert \nabla u\right\vert }\nabla u\right)  =0,\\
u|\partial\Omega=g
\end{array}
\right.
\]
where $a(s)=s^{p-1}A(s),$ $p>1,$ has a unique weak solution in $C^{0,1}\left(
\overline{\Omega}\right)  .$
\end{theorem}

\begin{proof}
The solution is obtained as a limit of a sequence $\left(  u_{\kappa_{i}%
}\right)  ,$ $\kappa_{i}>0,$ $\kappa_{i}\rightarrow0$ $\left(  i\rightarrow
\infty\right)  ,$ where $u_{\kappa_{i}}$ are solutions of the regularized
Dirichlet problems according to Corollary \ref{cor24} or Corollary
\ref{cormmi1}, respectively. By these corollaries, the families $\left(
u_{\kappa}\right)  $ are uniformly bounded in $C^{1}\left(  \overline{\Omega
}\right)  $ and hence, in particular, there is a uniform $L^{2}\left(
\overline{\Omega}\right)  $ bound for $\nabla u_{\kappa_{i}}$. We may
therefore find a function $u\in C^{0,1}\left(  \overline{\Omega}\right)  $
such that $u_{\kappa_{i}}$ converges (up to a subsequence) uniformly to $u$ in
$\overline{\Omega}$ and $\nabla u_{\kappa_{i}}$ converges weakly to $\nabla u$
in $L^{2}\left(  \Omega\right)  .$ The statement of the theorem follows
immediately if we can show that $u$ minimizes the functional%
\[
F(v)=\int_{\Omega}\phi\left(  \left\vert \nabla v\right\vert \right)  dV
\]
subject to the boundary condition $v|\partial\Omega=g,$ i.e. that $F(u)\leq
F(v)$ for any $v\in C^{1}\left(  \overline{\Omega}\right)  $ with
$v|\partial\Omega=g.$

We clearly have
\[
F_{\kappa}\left(  u_{\kappa}\right)  \leq F_{\kappa}\left(  v\right)
\]
for all such $v$ where%
\[
F_{\kappa}\left(  v\right)  =\int_{\Omega}\phi_{\kappa}\left(  \left\vert
\nabla v\right\vert \right)  dw,\text{ \ }\phi_{\kappa}\left(  s\right)
=\int_{0}^{s}a_{\kappa}\left(  t\right)  dt\geq\phi\left(  s\right)  .
\]
By a well known result in the Calculus of Variations, using the convexity of
$\phi$ \cite{M} we get for arbitrary $v\in C^{1}\left(  \overline{\Omega
}\right)  ,$ $v|\partial\Omega=g$:%
\[
F(u)\leq\lim\inf F\left(  u_{\kappa_{i}}\right)  \leq\lim\inf F_{\kappa_{i}%
}\left(  u_{\kappa_{i}}\right)  \leq\lim\inf F_{\kappa_{i}}\left(  v\right)
=F(v).
\]

\end{proof}

As announced above we shall now show using a result of Lieberman how one may
deduce the $C^{1,\beta}$ regularity in the singular or degenerate case. As a
consequence of this, the set where the gradient of the solution does not
vanish is open and hence the solution is of class $C^{2}$ on this set.

\begin{theorem}
\label{thh}We assume all the conditions of Theorem \ref{exreg} but with
$a(s)=s^{p-1}A(s),$ $p>1,$ and replacing (\ref{min1}) by (\ref{nmin1}). Then
there is $\beta>0$ such that the Dirichlet problem%
\begin{equation}
\left\{
\begin{array}
[c]{l}%
\operatorname{div}\left(  \frac{a\left(  \left\vert \nabla u\right\vert
\right)  }{\left\vert \nabla u\right\vert }\nabla u\right)  =0,\\
u|\partial\Omega=g
\end{array}
\right.  \label{pdelib}%
\end{equation}
has a unique weak solution $u\in C^{1,\beta}\left(  \overline{\Omega}\right)
$.
\end{theorem}

\begin{proof}
We shall obtain the solution of (\ref{pdelib}) as limit of solutions
$u_{\kappa_{i}}$ of the corresponding regularized equations \ when $\kappa
_{i}\rightarrow0$ as $i\rightarrow\infty.$ The existence of $u_{\kappa_{i}}$
and the uniform $C^{1}\left(  \overline{\Omega}\right)  $ bound is guaranteed
by Corollary \ref{cor24}. We shall show that Theorem 1 of \cite{L} may be
applied to obtain a uniform $C^{1,\beta}\left(  \overline{\Omega}\right)  $
bound for the sequence $u_{\kappa_{i}}.$ From this the statement of the
theorem is an immediate consequence. In order to meet the assumptions of
Lieberman we must modify our equation as follows. We replace the coefficient
$A$ by
\[
\widetilde{A}(s)=\left\{
\begin{array}
[c]{c}%
A(s)\text{ if }s\leq s_{0}\\
A(s_{0})\text{ if }s>s_{0}%
\end{array}
\right.
\]
where $s_{0}$ is an upper bound for the $C^{1}\left(  \overline{\Omega
}\right)  $ norm of our solution $u_{\kappa}$ which is guaranteed by Corollary
\ref{cor24}. It is then obvious that $u_{\kappa}$ will be a solution of the
modified equation
\[
\operatorname{div}\left(  \left(  \kappa+\left\vert \nabla u_{\kappa
}\right\vert ^{2}\right)  ^{\frac{p}{2}-1}\widetilde{A}\left(  \left\vert
\nabla u_{\kappa}\right\vert \right)  \nabla u_{\kappa}\right)  =0,
\]
in local coordinates (see (\ref{el}))
\[
\sum\frac{\partial}{\partial x_{j}}\widetilde{A}^{j}\left(  x,Du\right)  =0
\]
with%
\[
\widetilde{A}^{j}\left(  x,Du\right)  =\sum\sqrt{g}g^{ij}\left(
\kappa+\left\vert \nabla u\right\vert ^{2}\right)  ^{\frac{p}{2}-1}%
\widetilde{A}\left(  \left\vert \nabla u\right\vert \right)  \frac{\partial
u}{\partial x_{i}}.
\]
Let us remark that for a function $u$ on the manifold $\mathcal{M}$ the norm
of its gradient $\nabla u$ on $\mathcal{M}$ is equivalent to the Euclidean
norm of its Euclidean gradient $Du.$ Moreover the quotient%
\[
\frac{\kappa+s^{2}}{\left(  \sqrt{\kappa}+s\right)  ^{2}}%
\]
being bounded by positive constants from above and from below independent of
$0<\kappa\leq1$ and $s\geq0$ the term $\left(  \kappa+\left\vert \nabla
u\right\vert ^{2}\right)  ^{\frac{p}{2}-1}$ may be replaced in all estimates
by $\left(  \sqrt{\kappa}+\left\vert Du\right\vert \right)  ^{p-2}.$
Therefore, by Lemma \ref{nnn}, the ellipticity required in Lieberman's theorem
is satisfied. It remains to check the Holder estimates of the coefficients
$\widetilde{A}^{j}\left(  x,Du\right)  $ with respect to the coordinates
$x=\left(  x_{1},...,x_{n}\right)  $. We even show a Lipschitz condition by
computing a bound for the partial derivatives $\partial\widetilde{A}%
^{j}/\partial x_{k}.$

Denoting by $C$ a generic constant which only depends on the metric tensor
$g_{ij}$ and it first derivatives we get%
\[
\frac{\partial\left\vert \nabla u\right\vert ^{2}}{\partial x_{k}}\leq
C\left\vert Du\right\vert ^{2},\text{ }\frac{\partial\left\vert \nabla
u\right\vert }{\partial x_{k}}\leq C\left\vert Du\right\vert
\]
and hence%
\begin{align*}
\left\vert \frac{\partial\widetilde{A}^{j}}{\partial x_{k}}\left(
x,Du\right)  \right\vert  &  \leq C\left\{  \left(  \kappa+\left\vert \nabla
u\right\vert ^{2}\right)  ^{\frac{p}{2}-1}\widetilde{A}\left(  \left\vert
\nabla u\right\vert \right)  \left\vert Du\right\vert \right. \\
&  +\left(  \kappa+\left\vert \nabla u\right\vert ^{2}\right)  ^{\frac{p}%
{2}-2}\left\vert Du\right\vert ^{2}\widetilde{A}\left(  \left\vert \nabla
u\right\vert \right)  \left\vert Du\right\vert \\
&  \left.  +\left(  \kappa+\left\vert \nabla u\right\vert ^{2}\right)
^{\frac{p}{2}-1}\widetilde{A}^{\prime}\left(  \left\vert \nabla u\right\vert
\right)  \left\vert Du\right\vert ^{2}\right\}
\end{align*}
Setting
\[
K:=\max_{0\leq s\leq s_{0}}\left\{  A(s),A^{\prime}(s)\right\}  =\max_{0\leq
s\leq s_{0}}\left\{  \widetilde{A}(s),\widetilde{A}^{\prime}(s)\right\}
\]
and observing that $\widetilde{A}^{\prime}(s)=0$ for $s>s_{0}$ we obtain the
estimate
\begin{align*}
\left\vert \frac{\partial\widetilde{A}^{j}}{\partial x_{k}}\left(
x,Du\right)  \right\vert  &  \leq CK\left\{  \left(  1+\left\vert
Du\right\vert \right)  ^{p-1}+\left(  1+s_{0}^{2}\right)  ^{\frac{p}{2}%
-1}s_{0}^{2}\right\} \\
&  \leq\widetilde{C}\left(  1+\left\vert Du\right\vert \right)  ^{p-1}%
\end{align*}
with a constant $\widetilde{C}$ also depending on $s_{0}.$ This completes the proof.
\end{proof}

On the basis of Corollary \ref{cor23} the proof of the following theorem
proceeds along the same lines as that of the previous theorem by replacing
$\Omega$ by a relatively compact subdomain $\Lambda$ of $\Omega$ and using a
local version of the $C^{1,\beta}$ estimates in Lieberman's theorem. We note
that though such a local version is not explicitly stated in Lieberman's
result, it is obviously a necessary ingredient to get global estimates. Indeed
the interior estimates can be obtained by the same methods as the estimates at
the boundary.

\begin{theorem}
\label{thth}We assume all the conditions of Theorem \ref{excbd} but replacing
(\ref{min1}) by (\ref{nmin1}). Then there is a unique weak solution of the
Dirichlet problem%
\[
\left\{
\begin{array}
[c]{l}%
\operatorname{div}\left(  \frac{a\left(  \left\vert \nabla u\right\vert
\right)  }{\left\vert \nabla u\right\vert }\nabla u\right)  =0,\\
u|\partial\Omega=g
\end{array}
\right.
\]
which, on each subdomain $\Lambda$ of $\Omega$ with $\overline{\Lambda}%
\subset\Omega$, belongs to $C^{1,\beta}\left(  \overline{\Lambda}\right)  $
for some $\beta$ possibly depending on $\Lambda$. Moreover, for each such
relatively compact subdomain $\Lambda$ there is a $C^{1,\beta}\left(
\overline{\Lambda}\right)  $ bound for $u$ depending only on $\Lambda$ and
$\sup_{\Omega}\left\vert u\right\vert .$
\end{theorem}

\begin{theorem}
\label{singmin1}We assume all the conditions of Theorem \ref{excbd} but with
$a(s)=s^{p-1}A(s),$ $p>1,$ and replacing (\ref{mmin1}) by (\ref{nmin1}). Then
there is a unique weak solution of the Dirichlet problem%
\[
\left\{
\begin{array}
[c]{l}%
\operatorname{div}\left(  \frac{a\left(  \left\vert \nabla u\right\vert
\right)  }{\left\vert \nabla u\right\vert }\nabla u\right)  =0,\\
u|\partial\Omega=g
\end{array}
\right.
\]
which belongs to $C^{1,\beta}\left(  \overline{\Omega}\right)  .$
\end{theorem}

\begin{proof}
On the basis of Corollary \ref{cormmi1} the proof is completely analogous to
that of Theorem \ref{thh}.
\end{proof}

\begin{theorem}
\label{singmin2}We assume all the conditions of Theorem \ref{exminreg} but
with $a(s)=s^{p-1}A(s),$ $p>1,$ and replacing (\ref{mmin1}) by (\ref{nmin1}).
Then the Dirichlet problem%
\[
\left\{
\begin{array}
[c]{l}%
\operatorname{div}\left(  \frac{a\left(  \left\vert \nabla u\right\vert
\right)  }{\left\vert \nabla u\right\vert }\nabla u\right)  =0,\\
u|\partial\Omega=g
\end{array}
\right.
\]
has a unique weak solution $u\in C^{0}\left(  \overline{\Omega}\right)  \cap
C^{1}\left(  \Omega\right)  $ for any $g\in C^{0}\left(  \overline{\Omega
}\right)  .$ Moreover, for each relatively compact subdomain $\Lambda$ of
$\Omega$ the solution $u$ belongs to $C^{1,\beta}\left(  \overline{\Lambda
}\right)  $ for some $\beta>0$ and there is an a-priori bound for the
$C^{1,\beta}$ norm of $u$ depending only on $\sup_{\Omega}\left\vert
u\right\vert ,$ $\Lambda$ and $\Omega.$
\end{theorem}

We close this section by presenting an example where Theorems \ref{singmin1}
and \ref{singmin2} apply. This example can be seen as a $p-$area version of
the minimal surface equation, analogous to the $p-$energy of the Laplace
partial differential equation, $1<p\leq2.$

\begin{example}
We consider the integrand%
\[
\phi\left(  s\right)  =\left(  1+s^{p}\right)  ^{\frac{1}{p}}%
\]
for which we get%
\begin{align*}
a\left(  s\right)   &  =\phi^{\prime}\left(  s\right)  =\frac{s^{p-1}}{\left(
s^{p}+1\right)  ^{\frac{p-1}{p}}},\text{ }1+b(s)=\frac{sa^{\prime}(s)}%
{a(s)}=\frac{p-1}{s^{p}+1}\\
b^{\prime}(s)  &  =-\frac{p\left(  p-1\right)  s^{p-1}}{\left(  s^{p}%
+1\right)  ^{2}}.
\end{align*}
Hence we obtain%
\[
\left(  -sb^{\prime}(s)-\left(  1+\varepsilon\right)  \left(  1+b(s)\right)
\right)  s^{2}=\frac{p-1}{\left(  s^{p}+1\right)  ^{2}}\left(  \left(
p-1-\varepsilon\right)  s^{p+2}-\left(  1+\varepsilon\right)  s^{2}\right)
\geq\alpha>0
\]
for $s\geq s_{0}\left(  p,\varepsilon\right)  $ provided that $\varepsilon
<p-1$ and $p\leq2.$
\end{example}

\newpage

\section{The asymptotic Dirichlet problem}

\qquad The existence or nonexistence of non constant entire bounded harmonic
functions on a Cartan-Hadamard manifold $\mathcal{M}$ is a topic of study in
Differential Geometry that dates back to the 70's\ (see \cite{GW}, \cite{SY}).
In the last years this problem has been studied with other partial
differential equations such as the $p-$Laplacian (\cite{Ho}) and the minimal
surface equation (\cite{GR}, \cite{CHR}, \cite{CHR2}, \cite{RT}). The class of
partial differential equations considered here has been studied in \cite{RS}
with the purpose of proving Liouville type theorems that is, \emph{non
existence} of non constant bounded solutions. In this last part of our notes
we investigate the \emph{existence} of bounded non constant solutions of this
class of partial differential equations by studying the associated asymptotic
Dirichlet problem.

\subsection{Existence theorems}

\qquad A natural way of finding bounded entire solutions to a partial
differential equation on a Cartan-Hadamard manifold is by solving the
asymptotic Dirichlet problem with a prescribed non constant boundary data
given at infinity.

Recall that a Cartan-Hadamard manifold is a complete, connected and simply
connected Riemannian $n$-manifold $\mathcal{M}$, $n\geq2$, of non-positive
sectional curvature.\ By the Cartan-Hadamard theorem, the exponential map
$\exp_{o}\colon T_{o}\mathcal{M}\rightarrow\mathcal{M}$ is a diffeomorphism
for every point $o\in\mathcal{M}$. Consequently, $\mathcal{M}$ is
diffeomorphic to $\mathbb{R}^{n}$.

A Cartan-Hadamard manifold $\mathcal{M}$ can be compactified by adding a
\emph{sphere at infinity} which is also called the asymptotic boundary of
$\mathcal{M}$. The sphere at infinity of $\mathcal{M},$ denoted by
$\partial_{\infty}\mathcal{M},$ is defined as the set of all equivalence
classes $\left[  \gamma\right]  $ of unit speed geodesic rays $\gamma$ of
$\mathcal{M}$; two such rays $\gamma_{1}$ and $\gamma_{2}$ are equivalent if
$\sup_{t\geq0}d\left(  \gamma_{1}(t),\gamma_{2}(t)\right)  <\infty$. The
compactification $\overline{\mathcal{M}}$ of $\mathcal{M},$ also known as the
geometric compactification of $\mathcal{M},$ is then $\overline{\mathcal{M}%
}:=\mathcal{M}\cup\partial_{\infty}\mathcal{M},$ with the following topology.

Given $p\in\mathcal{M}$ let $B\subset\mathbb{S}^{n-1}$ be an open geodesic
ball of the unit sphere $\mathbb{S}^{n-1}$ of $T_{p}\mathcal{M}$. Given
$v\in\mathbb{S}^{n-1},$ denote by $\gamma_{v}:\left[  0,\infty\right)
\rightarrow\mathcal{M}$ the geodesic ray such that $\gamma_{v}\left(
0\right)  =p$ and $\gamma_{v}^{\prime}(0)=v.$ Then, given $t>0$ the sets%

\[
T=\left\{  \left[  \gamma_{v}\right]  \text{
$\vert$
}v\in B\right\}  \cup\left\{  \gamma_{v}\left(  \left(  t,\infty\right)
\right)  \text{
$\vert$
}v\in B\right\}
\]
with $p$ varying in $\mathcal{M}$, $B$ varying on the unit sphere of
$T_{p}\mathcal{M}$ and $t$ varying in the positive real numbers form a basis
for a topology on $\overline{\mathcal{M}},$ called the cone topology. The
space $\overline{\mathcal{M}},$ equipped with the cone topology, is
homeomorphic to a closed Euclidean ball. For more details see \cite{EO}.

The asymptotic Dirichlet problem on $\mathcal{M}$ for a differential operator
$Q$ on $\mathcal{M}$ consists in finding a (unique) function $u\in
C^{0}(\overline{\mathcal{M}})$ such that $Q\left[  u\right]  =0$ on
$\mathcal{M}$ and $u|\partial_{\infty}M=g,$ for a given function $g\in
C^{0}\left(  \partial_{\infty}\mathcal{M}\right)  $.

We consider on $\mathcal{M}$ the same family of operators already considered
here so far, namely, operators of the form
\begin{equation}
Q\left[  u\right]  =\operatorname{div}\left(  \frac{a(\lvert\nabla u\rvert
)}{\lvert\nabla u\rvert}\nabla u\right)  \label{M_equ}%
\end{equation}
where $a\in C^{2}\left(  \left(  0,\infty\right)  \right)  \cap C^{0}\left(
[0,\infty)\right)  $ satisfies conditions to be discussed on the course of the text.

For further references in the text, it is convenient to state the asymptotic
Dirichlet problem for $Q$ in the following short form%
\begin{equation}
\left\{
\begin{array}
[c]{l}%
Q\left[  u\right]  =\operatorname{div}\left(  \frac{a(\lvert\nabla u\rvert
)}{\lvert\nabla u\rvert}\nabla u\right)  =0\text{ on }\mathcal{M}\\
u|\partial_{\infty}\mathcal{M}=g,\text{ }u\in C^{1}(\mathcal{M})\cap
C^{0}(\overline{\mathcal{M}}).
\end{array}
\right.  \label{D}%
\end{equation}

We follow now closely the paper \cite{RT} of J. Ripoll and M. Telichevesky. In
the case of bounded domains, the continuous extension to the boundary of the
domain of a prospective solution to the Dirichlet problem of an elliptic
partial differential equation (for example, the one obtained by Perron's
method is typical, \cite{GT}) depends on the \textit{regularity }of the domain
with respect to the partial differential equation, that is, on the existence
of barriers at each point of the boundary of the domain (see Subsection
\ref{mapbs} and also \cite{GT}). To deal with the asymptotic Dirichlet problem
in $\mathcal{M}$ we extend this notion of regularity to the asymptotic
boundary $\partial_{\infty}\mathcal{M}$ of $\mathcal{M}$.

We consider here weak $C^{1}$ solutions to the equation $Q\left[  u\right]
=0$ in $\mathcal{M}$ that is, we require that $u\in C^{1}\left(
\mathcal{M}\right)  $ and satisfies
\begin{equation}
\int_{\mathcal{M}}\left\langle \frac{a\left(  \lvert\nabla u\rvert\right)
}{\lvert\nabla u\rvert}\nabla u,\nabla\zeta\right\rangle dx=0 \label{Qsol}%
\end{equation}
for every $\zeta\in C_{0}^{\infty}(\mathcal{M})$. A function $v\in
C^{1}\left(  \Omega\right)  $ is a subsolution of $Q$ in a domain $\Omega$ of
$\mathcal{M}$ if $Q[v]\geq0$ weakly in $\Omega$, that is
\begin{equation}
\int_{\mathcal{M}}\left\langle \frac{a\left(  \lvert\nabla v\rvert\right)
}{\lvert\nabla v\rvert}\nabla v,\nabla\zeta\right\rangle dx\leq0 \label{Qsub}%
\end{equation}
for every non-negative $\zeta\in C_{0}^{\infty}(\Omega)$. A function $w\in
C^{1}\left(  \Omega\right)  $ is called a supersolution of $Q$ in $\Omega$ if
$-w$ is a subsolution for $Q$ in $\Omega$.

Given $x\in\partial_{\infty}\mathcal{M}$ and an open subset $\Omega
\subset\mathcal{M}$ such that $x\in\partial_{\infty}\Omega$, an upper barrier
for $Q$ relative to $x$ and $\Omega$ with height $C$ is a function $w\in
C^{1}\left(  \mathcal{M}\right)  $ such that

\begin{description}
\item \textrm{(i)} $w$ is a supersolution for $Q$

\item \textrm{(ii) }$w\geq0$ and $\lim_{p\in\mathcal{M},\,p\rightarrow
x}w(p)=0$

\item \textrm{(iii)} $w_{\mathcal{M}\setminus\Omega}\geq C$.
\end{description}

\noindent Lower barriers are defined similarly.

We say that $\mathcal{M}$ is regular at infinity with respect to $Q$ if, given
$C>0$, $x\in\partial_{\infty}M$ and an open subset $W\subset\partial_{\infty
}M$ with $x\in W$, there exist an open set $\Omega\subset M$ such that
$x\in\operatorname*{Int}\partial_{\infty}\Omega\subset W$ and upper and lower
barriers $w,v\in C^{1}\left(  \mathcal{M}\right)  $ relative to $x$ and
$\Omega$, with height $C$.

The regularity at infinity has already been considered by other authors for
the $p-$Laplacian. The reader should compare the above definition with
Definition 2.6 and Theorem 2.7 in \cite{Choi} for the case of the Laplace
operator and also with Theorem 3.3 and Definition 3.4 in \cite{HV} for the
case of the $p-$Laplacian.

\begin{theorem}
\label{abcd}Let $\mathcal{M}$ be a Hadamard manifold which is regular at
infinity with respect to $Q.$ Assume moreover that

\begin{description}
\item[(a)] given $\phi\in C^{0}\left(  \mathcal{M}\right)  ,$ there is a
sequence of bounded $C^{2,\alpha}$ domains $\Omega_{k}\subset\mathcal{M},$
$k\in\mathbb{N},$ satisfying $\Omega_{k}\subset\Omega_{k+1},$ $\cup\Omega
_{k}=\mathcal{M}$ such that, given $k,$ there is a weak solution $u_{k}\in
C^{0}(\overline{\Omega_{k}})\cap C^{1}(\Omega_{k})$ of the Dirichlet problem
for $Q\left[  u\right]  =0$ in $\Omega_{k}$ such that $u_{k}|\partial
\Omega_{k}=\phi|\partial\Omega_{k}$

\item[(b)] sequences of solutions with uniformly bounded $C^{0}$ norm are
compact in the $C^{1}$ norm in precompact subsets of $\mathcal{M}$.
\end{description}

\noindent Then the asymptotic Dirichlet (\ref{D}) is solvable for any
continuous asymptotic boundary data $g\in C^{0}\left(  \partial\mathcal{M}%
_{\infty}\right)  $.
\end{theorem}

\begin{proof}
Let $\phi\in C^{0}(\overline{M})$ be a continuous extension to $\overline{M}$
of the asymptotic boundary data $g$ of problem (\ref{D}). From condition
\textbf{(a)} there is a solution $u_{k}\in C^{0}(\overline{\Omega_{k}})\cap
C^{1}(\Omega_{k})$ of the Dirichlet problem%
\[
\left\{
\begin{array}
[c]{l}%
\mathcal{Q}[u]=0\text{ in }\Omega_{k},\text{ }\\
u|\partial\Omega_{k}=\phi|\partial\Omega_{k}.
\end{array}
\right.
\]
Condition \textbf{(b)} together with the diagonal method show that there
exists a subsequence of $(u_{k})$ converging uniformly on compact subsets of
$\mathcal{M}$ in the $C^{1}$ norm to a global solution $u\in C^{1}%
(\mathcal{M})$ of $Q\left[  u\right]  =0$. From the comparison principle it
follows that
\[
\sup_{\mathcal{M}}\left\vert u\right\vert \leq\sup_{\mathcal{M}}\left\vert
\phi\right\vert .
\]

One needs to show that $u$ extends continuously to $\partial_{\infty
}\mathcal{M}$ and satisfies $u|\partial_{\infty}\mathcal{M}=g$. Let
$x\in\partial_{\infty}\mathcal{M}$ and $\varepsilon>0$ be given.

Since $g$ is continuous, there exists an open neighborhood $W\subset
\partial_{\infty}\mathcal{M}$ of $x$ such that $g(y)<g(x)+\varepsilon/2$ for
all $y\in W$. Furthermore, regularity of $\partial_{\infty}\mathcal{M}$
implies that exists an open subset $\Omega\subset\mathcal{M}$ such that
$x\in\operatorname*{Int}\left(  \partial_{\infty}\Omega\right)  \subset W$ and
$w:\mathcal{M}\rightarrow\mathbb{R}$ upper barrier with respect to $x$ and
$\Omega$ with height $C:=2\max_{\overline{M}}|g|$.

Defining
\[
v(p):=w(p)+g(x)+\varepsilon,\text{ }p\in\Omega,
\]
we claim that $u\leq v$ in $\Omega$.

From the continuity of $\phi$ we may find $k_{0}$ such that $\phi
(p)<g(x)+\varepsilon/2$ for all $p\in\partial\Omega_{k}\cap\Omega$, $k\geq
k_{0}$. Moreover, we may choose $k_{0}$ such that $\Omega_{k_{0}}\cap
\Omega\neq\emptyset$. Set $V_{k}:=\Omega\cap\Omega_{k}$, $k\geq k_{0}$. We
note that $u_{k}\leq v$ in $V_{k}$. Indeed, this inequality holds on
\[
\partial V_{k}=\overline{\left(  \partial\Omega_{k}\cap\Omega\right)  }%
\cup\overline{\left(  \partial\Omega\cap\Omega_{k}\right)  }.
\]
On $\partial\Omega_{k}\cap\Omega$ is due to the choice of $k_{0}$; on
$\partial\Omega\cap\Omega_{k}$ it holds because $w\geq\max|g|$ on
$\partial\Omega$, which implies that $w\geq u_{k}$, by the comparison
principle. Also the comparison principle implies that $u_{k}\leq v$ in $V_{k}%
$. Since it holds for all $k\geq k_{0}$, we have $u\leq v$ on $\Omega$.

It is also possible to define $v_{-}:\mathcal{M}\rightarrow\mathbb{R}$ by
$v_{-}(p):=\varphi(x)-\varepsilon-w(p)$ in order to obtain $u\geq v_{-}$ in
$\Omega$. We then have%
\[
\left\vert u(p)-\varphi(x)\right\vert <\varepsilon+w(p),\,
\]
for all $p\in\Omega$ and hence
\[
\limsup_{p\rightarrow x}\left\vert u(p)-\varphi(x)\right\vert \leq
\varepsilon.
\]
The proof is complete, since $\varepsilon>0$ is arbitrary.
\end{proof}

\bigskip

Theorem \ref{abcd} brings up the problem of when the Hadamard manifold
$\mathcal{M}$ is regular at infinity. The following definition, introduced in
\cite{RT}, turned out to be a key concept.

\begin{definition}
\label{sc}Let $\mathcal{M}$ be a Hadamard manifold. We say that $\mathcal{M}$
satisfies the \emph{strict convexity condition (SC condition)} if, given
$x\in\partial_{\infty}\mathcal{M}$ and a relatively open subset $W\subset
\partial_{\infty}\mathcal{M}$ containing $x,$ there exists a $C^{2}$ open
subset $\Omega\subset\overline{\mathcal{M}}$ such that $x\in
\operatorname*{Int}\left(  \partial_{\infty}\Omega\right)  \subset W,$ where
$\operatorname*{Int}\left(  \partial_{\infty}\Omega\right)  $ denotes the
interior of $\partial_{\infty}\Omega$ in $\partial_{\infty}\mathcal{M},$ and
$\mathcal{M}\setminus\Omega$ is convex.
\end{definition}

Informally speaking one may say that, as it happens with strictly convex
bounded domains of the Euclidean space, $\mathcal{M}$ satisfies the SC
condition when one can extract from $\overline{\mathcal{M}}$ a neighborhood of
any point of $\partial_{\infty}\mathcal{M}$ such that what remains is still convex.

\begin{lemma}
\label{barriers} Let $\mathcal{M}$ be a Hadamard manifold with sectional
curvature $K_{M}\leq-k^{2}<0$ and satisfying the SC condition. Assume that%
\begin{equation}
Q\left[  u\right]  =\mathrm{{\operatorname{div}}}\left(  \frac{a(|\nabla
u|)}{|\nabla u|}\nabla u\right)  =0 \label{Qdefinition}%
\end{equation}
with $a\in C^{0}\left(  \left[  0,\infty\right)  \right)  \cap C^{1}\left(
\left(  0,\infty\right)  \right)  $ satisfies
\begin{align}
&  a(0)=0,a^{\prime}(s)>0\text{ for all }s>0;\label{a1}\\
&  \text{there exist }q>0\text{ and }\delta>0\text{ such that }a(s)\geq
s^{q},\text{ }s\in\left[  0,\delta\right]  . \label{a3}%
\end{align}
Then $\mathcal{M}$ is regular at infinity with respect to $Q$.
\end{lemma}

\begin{proof}
Let $C>0$ and $x\in W\subset\partial_{\infty}\mathcal{M}$ be given. Since
$Q\left[  -u\right]  =-Q\left[  u\right]  $ it is enough to prove the
existence of barriers from above at $x.$ Since $\mathcal{M}$ satisfies the SC
condition, there exists a $C^{2}$ open subset $\Omega$ of $\mathcal{M}$ such
that $x\in\operatorname*{Int}\left(  \partial_{\infty}\Omega\right)  \subset
W$ and such that $\mathcal{M}\setminus\Omega$ is convex. Let $s:\Omega
\rightarrow\mathbb{R}$ be the distance function to $\partial\Omega$. Since
$\mathcal{M}\setminus\Omega$ is convex and $K_{M}\leq-k^{2}$, $k>0,$ we may
apply comparison theorems (see Theorems 4.2 and 4.3 of \cite{Choi}) to obtain
the estimate
\begin{equation}
\Delta s\geq(n-1)k\tanh ks. \label{Deltas}%
\end{equation}
On the other hand, since $a^{\prime}>0$, $a$ has an inverse function
$a^{-1}\in C^{1}\left(  \left[  0,\alpha\right)  \right)  $ where $\alpha=\sup
a\leq\infty$. Set
\[
c=\frac{a(2C)}{\cosh^{n-1}k}.
\]

Since $0\leq c\cosh^{1-n}kt\leq\alpha$ for all $t\geq0$ we may define a
function $g:\left[  0,\infty\right)  \rightarrow\mathbb{R},$ $g\in
C^{2}\left(  \left(  0,\infty\right)  \right)  ,$ possibly with $g(0)=\infty,$
by
\begin{equation}
g(s):=\int_{s}^{\infty}a^{-1}\left(  c\cosh^{1-n}kt\right)  dt. \label{g}%
\end{equation}

Without loss of generality, we may choose $\delta$ small such that $a\left(
\delta\right)  <c.$ Then, since
\[
c\cosh^{1-n}0=c>a(\delta)
\]
and $\lim_{t\rightarrow+\infty}c\cosh^{1-n}kt=0$ there is $\tau$ satisfying
\[
c\cosh^{1-n}k\tau=a\left(  \delta\right)  .
\]

Since for $t\in\left[  0,a\left(  \delta\right)  \right]  $ we have
$a^{-1}(t)\in\left[  0,\delta\right]  $ it follows from (\ref{a3}) that
\[
t=a\left(  a^{-1}\left(  t\right)  \right)  \geq a^{-1}(t)^{q}%
\]
and hence $a^{-1}(t)\leq t^{1/q}$ for $t\in\lbrack0,a\left(  \delta\right)
]$. Therefore
\begin{align*}
g(s)  &  \leq\int_{0}^{\tau}a^{-1}(c\cosh^{1-n}kt)dt+\int_{\tau}^{+\infty
}a^{-1}(c\cosh^{1-n}kt)dt\\
&  \leq a^{-1}(c)\tau+\int_{\tau}^{+\infty}(c\cosh^{1-n}kt)^{\frac{1}{q}}dt\\
&  \leq a^{-1}(c)\tau+(2c)^{\frac{1}{q}}\int_{\tau}^{+\infty}e^{-\frac{kt}{q}%
}dt=a^{-1}(c)\tau+\frac{(2c)^{\frac{k}{q}}}{q}e^{-k\tau}<+\infty
\end{align*}
for all $s>0$. Furthermore,
\[
g(0)>\int_{0}^{1}a^{-1}\left(  c\cosh^{1-n}kt\right)  dt\geq a^{-1}%
(c\cosh^{1-n}k)=2C
\]
and $\lim_{s\rightarrow\infty}g(s)=0$. Therefore we define $v:\Omega
\rightarrow\mathbb{R}$ as
\[
v(p):=g(s(p)),
\]
and will prove that $Q(v)\leq0$. We have%
\[
\nabla v(p)=g^{\prime}(s(p))\nabla s(p)=-a^{-1}\left(  c\cosh^{1-n}%
ks(p)\right)  \nabla s
\]
and then
\[
|\nabla v|=|g^{\prime}(s)|=a^{-1}\left(  c\cosh^{1-n}ks\right)  .
\]
Also $\nabla v/|\nabla v|=-\nabla s$. Combining the previous expressions, we obtain%

\begin{align*}
Q\left[  v\right]   &  =\operatorname{div}\mathrm{{\,}}\left(  -a\left(
|g^{\prime}(s)|\right)  \nabla s\right) \\
&  =\operatorname{div}\mathrm{{\,}}\left(  -a\left(  a^{-1}\left(
c\cosh^{1-n}ks\right)  \right)  \nabla s\right) \\
&  =\operatorname{div}\left(  -c\cosh^{1-n}ks\nabla s\right) \\
&  =-(1-n)ck\left(  \cosh^{-n}ks\right)  \left(  \sinh ks\right)
\langle\nabla s,\nabla s\rangle-c\left(  \cosh^{1-n}ks\right)  \Delta s\\
&  \leq(n-1)ck\left(  \cosh^{-n}ks\right)  \left(  \sinh ks\right) \\
&  -(n-1)c\left(  \cosh^{1-n}ks\right)  k\left(  \tanh ks\right)  =0,
\end{align*}
and hence, by Lemma \ref{QvleQu}, $v$ is a supersolution on $\Omega$.

To finish with the proof, define the global supersolution $w\in C^{0}\left(
\overline{\mathcal{M}}\right)  $ by
\[
w(p)=\left\{
\begin{array}
[c]{ll}%
\min\left\{  v(p),C\right\}  & \text{if }p\in\Omega\\
C & \text{if }p\in\overline{\mathcal{M}}\setminus\Omega,
\end{array}
\right.
\]
which is of course an upper barrier relative to $x$ and $\Omega$ with height
$C$.
\end{proof}

As a consequence of Theorems \ref{abcd} and \ref{barriers} we obtain

\begin{theorem}
\label{scab} Let $\mathcal{M}$ be a Hadamard manifold with sectional curvature
satisfying $K_{M}\leq-k^{2}<0$. Assume that $\mathcal{M}$ satisfies the SC
condition together with the conditions (\ref{a1}) and (\ref{a3}) of Lemma
\ref{barriers}, and that $Q$ satisfies conditions \emph{\textbf{(a)}} and
\emph{\textbf{(b) }}of Theorem \ref{abcd}. Then the asymptotic Dirichlet
problem (\ref{D}) is solvable for any $g\in C^{0}\left(  \partial_{\infty
}\mathcal{M}\right)  $.
\end{theorem}

\begin{theorem}
\label{coradp1}We assume that the operator%
\[
Q\left[  u\right]  =\operatorname{div}\frac{a\left(  \left\vert \nabla
u\right\vert \right)  }{\left\vert \nabla u\right\vert }\nabla u
\]
satisfies the following conditions (i) and \textbf{either }(ii) \textbf{or }(ii'):

(i) There are numbers $p>1$ and $\alpha\in\left(  0,1\right)  $ such that
$a(s)=s^{p-1}A(s)$ with $A\in C^{1,\alpha}\left(  \left[  0,\infty\right)
\right)  \cap C^{2,\alpha}\left(  \left(  0,\infty\right)  \right)  \ $and%
\[
\min_{0\leq s\leq\sigma}\left\{  A(s),\min\left\{  1,p-1\right\}
+\frac{sA^{\prime}(s)}{A(s)}\right\}  >0
\]
for any $\sigma>0$

(ii) There are numbers $\beta>0,$ $\gamma>0$ and $s_{0}>0$ such that with%
\begin{align*}
b(s)  &  :=\frac{sA^{\prime}(s)}{A(s)}\\
B(s)  &  :=\max\left\{  1,1+b(s)\right\}
\end{align*}
there holds
\[
B(s)^{-1}\left(  1+b(s)-\beta\left\vert b^{\prime}(s)\right\vert s\right)
\geq\gamma
\]
for $s\geq s_{0}.$

(ii') There are numbers $\varepsilon>0,$ $\gamma>0$ and $s_{0}>0$ such that
\[
\left(  -sb^{\prime}(s)-\left(  1+\varepsilon\right)  \left(  1+b(s)\right)
\right)  s^{2}\geq\gamma
\]
for $s\geq s_{0}.$

Then, if the Cartan-Hadamard manifold $\mathcal{M}$ has sectional curvature
$K_{\mathcal{M}}\leq-k^{2}$ for some number $k>0$ and satisfies the SC
condition (Definition \ref{sc}), the asymptotic Dirichlet problem%
\[
\left\{
\begin{array}
[c]{l}%
Q\left[  u\right]  =0\text{ on }\mathcal{M}\\
u|\partial_{\infty}\mathcal{M}=g
\end{array}
\right.
\]
has a unique solution $u\in C^{0}\left(  \overline{\mathcal{M}}\right)  \cap
C^{1}\left(  \mathcal{M}\right)  $ which, for each relatively compact
subdomain $\Omega\subset\mathcal{M}$ belongs to $C^{1,\lambda}\left(
\Omega\right)  $ with $\lambda>0$ possibly depending on $\Omega.$ In case
$p=2$ in (i) above the solution $u$ is classical, i.e., $u\in C^{2,\alpha
}\left(  \mathcal{M}\right)  .$
\end{theorem}

\begin{proof}
We observe that the uniqueness follows immediately from Proposition
\ref{QvleQu}. For the existence part, choose a fixed point of $\mathcal{M}$,
say $o,$ and, given $k\geq1,$ let $\Omega_{k}$ be the geodesic ball of
$\mathcal{M}$ centered at $o$ and with radius $k.$ Since $\mathcal{M}$ is a
Cartan Hadamard manifold $\Omega_{k}$ is a $C^{2,\alpha}$ domain for all $k$
and, from the Hessian comparison theorem, $\Omega_{k}$ is a convex and hence a
mean convex domain for all $k\geq1$ (Theorems 4.2 and 4.3 of \cite{Choi}).
Conditions (a) and (b) in Theorem \ref{abcd} are satisfied as follows from the
Theorems\ \ref{thth} and \ref{singmin2}. The conditions \ref{a1} and \ref{a3}
of Lemma \ref{barriers} clearly hold by assumption (i). Thus, Theorem
\ref{coradp1} is an immediate consequence of Theorem \ref{scab}.
\end{proof}

\subsection{Final comments}

\qquad Let $\mathcal{M}$ be a Hadamard manifold. If $\dim\mathcal{M}=2$ and
the sectional curvature of $\mathcal{M}$ satisfies $K_{\mathcal{M}}\leq
-k^{2},$ $k>0,$ then, since any two points at infinity of $\mathcal{M}$ can be
connected by a geodesic, it trivially follows that $\mathcal{M}$ satisfies the
SC condition. In arbitrary dimensions it is proved in \cite{RT} that if
$K_{\mathcal{M}}\leq-k^{2}$ and either $\mathcal{M}$ is a rotationally
symmetric (see \cite{Choi}) or the sectional curvature of $\mathcal{M}$ decays
at most exponentially then $\mathcal{M}$ satisfies the SC condition. The SC
condition also holds in Hadamard manifolds where the sectional curvature may
go to zero with a certain rate, but assuming stronger assumptions on the decay
of the curvature. The asymptotic Dirichlet problem can be solved in some of
these manifolds by using barriers at infinity others than those of Lemma
\ref{barriers} (see \cite{CHR3}).

We finally should remark that in dimensions greater than or equal to $3$ just
an upper bound for the sectional curvature of $\mathcal{M}$ is not enough for
the solvability of the asymptotic Dirichlet problem in $\mathcal{M}.$ Indeed,
Ancona \cite{Anc} and Borbely \cite{B2} construct examples of $3-$dimensional
Hadamard manifolds with curvature less than or equal to $-1$ such that if an
entire harmonic function extends continuously to the asymptotic boundary then
it is constant. Holopainen extended Borbely's example to the $p-$Laplacian
\cite{H2} and in \cite{HR} the authors constructed similar counter-examples of
manifolds for a class of partial differential equations that includes the
$p-$Laplacian and minimal surface equation. Therefore, in these manifolds,
because of Theorem \ref{scab}, the SC condition is not satisfied.

\end{document}